\documentclass[10pt,letterpaper,twoside]{article}

\usepackage[letterpaper,left=1in,right=1in,top=1.5in,bottom=1.5in]{geometry}

\usepackage[T1]{fontenc}
\usepackage{amsmath,amssymb,amsthm,mathtools}
\usepackage{eucal,mathrsfs}
\usepackage[bbgreekl]{mathbbol}
\usepackage{microtype}
\usepackage{verbatim}
\usepackage{tikz-cd}
\usepackage{enumitem}
\usepackage{xcolor}
\usepackage{fancyhdr}

\setlist{noitemsep}

\numberwithin{equation}{section}

\providecommand{\MR}[1]{}

\definecolor{todo}{rgb}{1,0,0}
\definecolor{conditional}{rgb}{0,1,0}
\definecolor{e-mail}{rgb}{0,.40,.80}
\definecolor{reference}{rgb}{.20,.60,.22}
\definecolor{mrnumber}{rgb}{.80,.40,0}
\definecolor{citation}{rgb}{0,.40,.80}

\let\oldmarginpar\marginpar
\renewcommand\marginpar[1]{\-\oldmarginpar[\raggedleft\footnotesize #1]%
  {\raggedright\footnotesize #1}}

\newcommand{\stack}[1]{\mathcal{#1}}
\newcommand{\sM}{\stack{M}}
\newcommand{\sE}{\stack{E}}
\newcommand{\sC}{\stack{C}}
\newcommand{\sF}{\stack{F}}
\newcommand{\sR}{\stack{R}}

\newcommand{\sheaf}[1]{\mathscr{#1}}
\newcommand{\OO}{\sheaf{O}}
\newcommand{\LL}{\sheaf{L}}
\newcommand{\VV}{\sheaf{V}}
\newcommand{\NN}{\sheaf{N}}

\newcommand{\bA}{\mathbf{A}}

\newcommand{\bG}{\mathbf{G}}
\newcommand{\bP}{\mathbf{P}}
\newcommand{\bQ}{\mathbf{Q}}
\newcommand{\bZ}{\mathbf{Z}}

\newcommand{\SB}{\mathbf{SB}}
\newcommand{\inv}{\mathrm{inv}}
\newcommand{\nr}{\mathrm{nr}}
\newcommand{\Aut}{\mathrm{Aut}}
\newcommand{\Pic}{\mathrm{Pic}}
\newcommand{\Gal}{\mathrm{Gal}}
\newcommand{\Br}{\mathrm{Br}}
\newcommand{\Hilb}{\mathrm{Hilb}}

\newcommand{\ind}{\mathrm{ind}}
\renewcommand{\deg}{\mathrm{deg}}

\newcommand{\red}{\mathrm{red}}

\newcommand{\Gm}{\bG_{m}}
\newcommand{\GL}{\mathbf{GL}}
\newcommand{\PGL}{\mathbf{PGL}}
\newcommand{\SL}{\mathbf{SL}}

\DeclareMathOperator{\Spec}{Spec}

\newcommand{\cdual}[1]{\widehat{#1}}

\newcommand{\jj}{\boldsymbol{j}}

\providecommand{\xlongrightarrow}[1]{\xrightarrow{#1}}
\renewcommand{\geq}{\geqslant}
\renewcommand{\leq}{\leqslant}

\providecommand{\Bar}[1]{\overline{#1}}
\newcommand{\kbar}{\Bar{k}}

\theoremstyle{plain}
\newtheorem{theorem}{Theorem}[section]
\newtheorem*{theorem*}{Theorem}
\newtheorem{lemma}[theorem]{Lemma}

\newtheorem{proposition}[theorem]{Proposition}

\newtheorem{corollary}[theorem]{Corollary}
\newtheorem*{corollary*}{Corollary}

\theoremstyle{definition}
\newtheorem{definition}[theorem]{Definition}

\newtheorem*{question*}{Question}

\newtheorem{remark}[theorem]{Remark}
\newtheorem{claim}[theorem]{Claim}

\title{Splitting Brauer classes by genus $1$ curves over number fields}
\author{Benjamin Antieau, Asher Auel, and Federico Scavia}
\date{}

\usepackage[
  pdfstartview=FitH,
  pdfauthor={Benjamin Antieau, Asher Auel, and Federico Scavia},
  pdftitle={Splitting Brauer classes by genus 1 curves over number fields},
  colorlinks,
  linkcolor=reference,
  citecolor=citation,
  urlcolor=e-mail,
  backref=page
]{hyperref}

\usepackage{cleveref}

\begin{document}

\maketitle
		
\begin{abstract}
\noindent We prove that every Severi--Brauer variety of dimension at least $2$ over a number field contains a twisted elliptic normal curve. Consequently, every Brauer class over a number field is split by a genus $1$ curve. We prove this using the fibration method, by showing that the rational points on a smooth compactification of the Hilbert scheme of twisted elliptic normal curves are dense in its Brauer--Manin set.
\end{abstract}

\vspace*{4mm}

\section{Introduction}

\paragraph{Splitting fields of Brauer classes.} One of the basic ways of studying Brauer classes is through their splitting fields, that is, field extensions over which they become trivial. Classically, one first considers finite splitting fields. For central simple algebras, this perspective leads to the study of maximal subfields and crossed products, to the index of a Brauer class, and to long-standing problems such as the period-index problem and Albert’s cyclicity problem. Splitting fields of geometric origin, namely function fields of smooth projective geometrically integral varieties, have also played a central role. They appear already in the use of Severi--Brauer varieties in the construction of generic splitting fields by Amitsur~\cite{amitsur1955generic}, the proofs of Merkurjev's theorem \cite{merkurjev1981norm} and the Merkurjev--Suslin theorem \cite{merkurjev1982k-cohomology}, and later for Galois cohomology in arbitrary degrees in the use of Pfister quadrics and Rost varieties in the proof by Voevodsky and Rost of the norm-residue isomorphism theorem (the Milnor conjecture and the Bloch--Kato conjecture), see \cite{haesemeyer2019norm}.
    
It is thus natural to ask about the arithmetic and geometry of $k$-varieties whose function fields split a given Brauer class $\alpha \in \Br(k)$. Recall that if $X$ is a Severi--Brauer variety representing $\alpha \in \Br(k)$, then $\alpha$ is split by the function field $k(X)$; more precisely, Amitsur proved that $X$ is universal with property in that for every field extension $L/k$, we have $\alpha_L=0$ in $\Br(L)$ if and only if $X$ has an $L$-point. In particular, $\alpha$ is split by the function field $k(Y)$ of an integral $k$-variety $Y$ if and only if there exists a rational map $Y \dashrightarrow X$.  One consequence is that $\alpha$ can be split by smooth projective geometrically integral varieties of dimension one; indeed, by Bertini's theorem, every Severi--Brauer variety contains a smooth projective geometrically integral curve $C$, as a complete intersection of anticanonical divisors, and hence in particular the Brauer class of $X$ is split by $k(C)$. This construction, however, typically produces curves of large genus. This is expected, because for every fixed genus $g\neq 1$ there is a basic obstruction to splitting Brauer classes by curves of genus~$g$. Indeed, if a smooth projective geometrically integral curve $C$ of genus $g\geq 0$ splits $\alpha$, then
\[
        \ind(\alpha)\mid \deg(\omega_C)=2g-2 .
\]
Consequently, no fixed genus $g\neq 1$ can split Brauer classes of arbitrary index. The exceptional case is $g=1$, where $2g-2=0$ and the obstruction disappears. Genus $1$ is therefore the only genus in which one can possibly hope for a splitting result for all Brauer classes.

\paragraph{The question of Clark and Saltman.} 

By definition, a \emph{genus $1$ curve} is a smooth projective geometrically connected curve of geometric genus $1$ over a field. The following question, due to Clark~\cite{clark:open} and Saltman~\cite{rage:problems}, arises naturally: for a field $k$ and a Brauer class $\alpha\in \Br(k)$, does there exist a genus $1$ curve $C$ over $k$ such that $\alpha_{k(C)}=0$ in $\Br(k(C))$? Equivalently, if $X$ is a Severi--Brauer variety representing $\alpha$, does there exist a morphism $C\to X$ from a genus $1$ curve $C$ over $k$?

This question sits at the intersection of several lines of research in the arithmetic of Brauer groups and genus $1$ curves, and has been studied through a variety of techniques: geometric constructions, duality pairings in Galois cohomology, Weil--Ch\^atelet groups, moduli stacks, gerbes, and twisted sheaves, among others.  Geometric constructions of genus 1 curves mapping to Severi--Brauer varieties of small dimension originate in the work of Artin~\cite{artin:Brauer-Severi} on the geometry of Severi--Brauer varieties, and made more explicit by work of Swets~\cite{swets1995global} and de Jong--Ho~\cite{dejong-ho}.  From the point of view of the period-index problem (for Brauer classes and for genus $1$ curves), the obstruction to realizing divisor classes on genus $1$ curves is governed by the period-index obstruction map of O'Neil \cite{oneil2001jacobians} and Clark \cite{clark2005period}, which the first two authors of the present article leveraged to show that, if $E$ is an elliptic curve over $k$ with full level $n$ structure, then every cyclic class in $\Br(k)[n]$ is split by an $E$-torsor (that is, a genus $1$ curve with Jacobian $E$); see~\cite[Theorem~C]{antieau-auel}. In a different direction, Ho and Lieblich~\cite{ho-lieblich} proved that every Brauer class is split by a torsor under an abelian variety; see also \cite[Theorem~B]{antieau-auel} for refinements. For a fixed genus $1$ curve $C$ over $k$, the relative Brauer group $\ker(\Br(k)\to \Br(C))$ has been studied by Roquette \cite{roquette1966splitting}, Lichtenbaum \cite{lichtenbaum1968period}, Ciperiani--Krashen \cite{ciperiani-krashen}, and Krashen--Lieblich \cite{krashen2008index}, among others. These works connect the Clark--Saltman question to the elementary obstruction to rational points, index-reduction formulas, and moduli of twisted sheaves. Finally, the finer problem of which Brauer classes are split by torsors under a prescribed elliptic curve has led to Saltman's obstructions \cite{saltman-genus} and the subsequent work of Mackall--Rekuski~\cite{mackall-rekuski}. 

Over an arbitrary field $k$ and for an arbitrary Brauer class $\alpha\in \Br(k)$, the Clark--Saltman question has a positive answer in the case $\mathrm{ind}(\alpha)\leq 3$ by Swets~\cite{swets1995global}, $\mathrm{ind}(\alpha)=4,5$ by de Jong--Ho~\cite{dejong-ho}, and by unpublished work of the second author if $\mathrm{ind}(\alpha)=6$.  In contrast, Reichstein and the third author \cite{reichstein2025brauer} constructed examples of Brauer classes of high index not split by any genus $1$ curve. More precisely, they proved that for every prime $p$, every integer $r$ with $r\geq 6$ if $p$ is odd and $r\geq 7$ if $p=2$, and every field $k_0$ containing a primitive $p$th root of unity, the Brauer class $\alpha=(t_1,t_2)_p+\cdots+(t_{2r-1},t_{2r})_p$ over $k=k_0(\!(t_1)\!)\cdots(\!(t_{2r})\!)$, corresponding to the central simple algebra given by the tensor product of cyclic algebras, is not split by any genus $1$ curve. More generally, their examples are not split by torsors under abelian varieties of dimension $g$ sufficiently small compared to $r$.

\paragraph{The Clark--Saltman question for cyclic Brauer classes.} The Brauer classes considered in \cite{reichstein2025brauer} have period $p$ and index $p^r$, where $r\geq 6$, and hence in particular are not cyclic. In fact, although the question of Clark--Saltman has a negative answer in general, there is evidence for a positive answer for \emph{cyclic} Brauer classes, and in particular over \emph{global fields}. (Recall that a Brauer class $\alpha\in \Br(k)[n]$ is cyclic if there exist $\chi\in H^1(k,\bZ/n\bZ)$ and $b\in H^1(k,\mu_n)$ such that $\alpha=\chi\cup b$ in $H^2(k,\mu_n)=\Br(k)[n]$; every Brauer class over a local or global field is cyclic.) Indeed, using work of Roquette~\cite{roquette1966splitting}, the first and second author showed that every Brauer class over a local field $k$ is split by a genus $1$ curve; see~\cite[Example~2.4 and Proposition~3.9]{antieau-auel} for the case when $k$ is non-archimedean. Moreover, by the aforementioned \cite[Theorem~C]{antieau-auel}, the Clark--Saltman question has a positive answer for cyclic classes in $\Br(k)[n]$ if there exists an elliptic curve $E$ over $k$ with full level $n$ structure. Finally, the question of Clark--Saltman has a positive answer if $k$ is a global field and $n=7$, and also for $n\in\{8,9,10,12\}$ under additional assumptions on the roots of unity in $k$; see \cite{antieau-auel}. In particular, the Clark--Saltman question over $k=\bQ$ is known to have a positive answer for Brauer classes of index $\leq 7$. 

There is also a very suggestive geometric construction, attributed to Michael Artin, which further supports this expectation. Let $\alpha$ be a cyclic Brauer class over a field $k$, choose a cyclic Galois algebra $L/k$ of degree $d\geq 3$ splitting $\alpha$, and let $X$ be a Severi--Brauer variety representing $\alpha$ of dimension $d-1$. After identifying $X_L$ with $\bP^{d-1}_L$ in such a way that the Galois action cyclically permutes the coordinates up to scalar multiplication, the standard $d$-gon formed by the coordinate lines descends to a nodal curve of arithmetic genus $1$ in $X$. If this nodal curve could be smoothed over $k$, one would obtain the desired genus $1$ curve in $X$. It is a folklore result, see \cite[Remark~3.3]{mackall:enc}, that this strategy can be made to work over large fields; see Proposition~\ref{nodal} and Remark~\ref{rmk:large}. Recall that a field $k$ is large if every curve over $k$ with a smooth $k$-point has a Zariski-dense set of $k$-points; for example, local fields, Laurent series fields, and $p$-closed fields (for some prime $p$) are large, see \cite{pop:large}. In particular, over a general field $k$, for every prime $p$, every cyclic class in $\Br(k)$ is split by a genus $1$ curve after passing to a suitable finite field extension of $k$ of prime-to-$p$ degree. However, global fields are not large, and so this strategy does not apply to them. 

\paragraph{Main results.}

Our main result is a positive answer to the Clark--Saltman question over all number fields. 

\begin{theorem}\label{maincor}
Let $k$ be a number field. For every Brauer class $\alpha\in \Br(k)$, there exists a genus $1$ curve $C$ over $k$ such that $\alpha_{k(C)}=0$ in $\Br(k(C))$.
\end{theorem}

Let $X$ be a Severi--Brauer variety of dimension $n-1\geq 2$ which represents the Brauer class $\alpha$. Recall that $\alpha$ is split by a genus $1$ curve $C$ over $k$ if and only if there exists a morphism $C\rightarrow X$ over $k$. We prove Theorem~\ref{maincor} by establishing the stronger result of existence of a \emph{twisted elliptic normal curve} in $X$, that is, a genus $1$ curve $C\subset X$ such that $C_{\kbar}\subset X_{\kbar}\cong \bP_{\kbar}^{n-1}$ is an elliptic curve of degree $n$ not contained in any hyperplane, equivalently, $C_{\kbar}$ is embedded in $\bP_{\kbar}^{n-1}$ by a complete linear system of degree $n$. 
	
\begin{theorem}\label{mainthm}
Let $k$ be a number field and let $X$ be a Severi--Brauer variety of dimension $\geq 2$ over $k$. There exists a twisted elliptic normal curve $C\subset X$ over $k$. 
\end{theorem}	
    
The key to the proof of Theorem~\ref{mainthm} is to view it as a statement about existence of rational points on a Hilbert scheme. Indeed, for a fixed central simple $k$-algebra $A$, the twisted elliptic normal curves in the Severi--Brauer variety $\SB(A)$ form a quasi-projective $k$-variety $V(n)^A$, and Theorem~\ref{mainthm} asserts that this Hilbert scheme has a $k$-point whenever $k$ is a number field. An important feature of the argument is that we do not fix the Jacobian of the genus $1$ curve in advance, but we allow it to move in the moduli stack $\sM_{1,1}$. Through the fibration method, this makes it possible to overcome the Brauer--Manin obstructions to existence of rational points that may appear on individual fibers.  We remark that the known geometric constructions, see \cite{dejong-ho}, which prove Theorem~\ref{mainthm} for Severi--Brauer varieties of dimension $2,3,4$ over every field, actually show that the Hilbert scheme $V(n)^A$ is unirational over $k$, hence contains a Zariski dense set of $k$-points.

We also prove a local-global principle for the existence of twisted elliptic normal curves whose Jacobian is a prescribed elliptic curve $E$ over $k$, under a large Galois image assumption on the torsion subgroup $E[2n](\kbar)$.
	
\begin{theorem}\label{local-global-thm}
Let $k$ be a number field, let $\Gamma_k$ be the absolute Galois group of $k$, let $X$ be a Severi--Brauer variety of dimension $n-1\geq 2$ over $k$, let $E$ be an elliptic curve  over $k$ such that the image of the natural homomorphism $\Gamma_k\to \GL(E[2n](\kbar))$ contains $\SL(E[2n](\kbar))$, and let $j_0=\jj(E)\in k$ be the $j$-invariant of $E$. 
\begin{itemize}
\item[{\em (1)}] If, for every place $v$ of $k$, there exists a twisted elliptic normal curve $C_v\subset X_{k_v}$ such that $\Pic^0_{C_v/k_v}\cong E_{k_v}$ over $k_v$, then there exists a twisted elliptic normal curve $C\subset X$ such that $\Pic^0_{C/k}\cong E$ over $k$.

\item[{\em (2)}] If, for every place $v$ of $k$, there exists a twisted elliptic normal curve $C_v\subset X_{k_v}$ such that $\jj(\Pic^0_{C_v/k_v})=j_0$, then there exists a twisted elliptic normal curve $C\subset X$ such that $\jj(\Pic^0_{C/k})=j_0$.
\end{itemize}
\end{theorem}
	
The large-image condition in Theorem~\ref{local-global-thm} holds for most elliptic curves over $k$, in the sense of Hilbert's irreducibility theorem.

Finally, the same circle of ideas gives not only existence of rational points, but weak approximation. Recall that a $k$-variety $X$ is said to satisfy \emph{weak approximation} if for every finite set $S$ of places of $k$, the image of the diagonal map $X(k)\to \prod_{v \in S} X(k_v)$ is dense. Let $V(n)^A$ be the Hilbert scheme of twisted elliptic normal curves in the Severi--Brauer variety of $A$, as defined in Section~\ref{twisted-case}.
    		
\begin{theorem}\label{weak-approximation-thm}
For every number field $k$ and every central simple algebra $A$ of degree $n\geq 3$ over $k$, the $k$-variety $V(n)^A$ satisfies weak approximation.
\end{theorem}

\paragraph{Proof sketch of Theorem~\ref{maincor}}

We recall that Theorem~\ref{maincor} is a direct corollary of Theorem~\ref{mainthm}, which we now sketch our proof of.  Let $A$ be a central
simple algebra of degree $n\geq 3$ over a number field $k$, and let
$\SB(A)$ be its Severi--Brauer variety. We let $V(n)^A$ be the open subscheme
of the Hilbert scheme of $\SB(A)$ parametrizing twisted elliptic normal
curves.  Thus $V(n)^A(k)\neq \emptyset$ if and only if $\SB(A)$ contains a twisted elliptic normal curve over $k$.

Let $\sM_{1,1}$ be the moduli stack of elliptic curves over $k$. Taking the Jacobian of the universal twisted elliptic normal curve over $V(n)^A$ gives a morphism $J^A\colon V(n)^A\to\sM_{1,1}$. For every elliptic curve $E$ over $k$, the fiber $V_E^A$ of $J^A$ over $E$ parametrizes twisted elliptic normal curves $C$ in $\SB(A)$ with Jacobian $E$. The translation action of $E[n]$ on $C$ extends to the tautological action of $\PGL(A)$ on $\SB(A)$, and this makes $V_E^A$ into a $\PGL(A)$-homogeneous space with geometric stabilizer isomorphic to $E[n]_{\kbar}$. More strongly, $J^A$ is a family of homogeneous spaces over $\sM_{1,1}$ for twisted forms of $\PGL_n$ with geometric stabilizer isomorphic to $\bZ/n\bZ\times_{\Bar{k}} \mu_n$; see Propositions~\ref{prop:pgl-homogeneous-a} and~\ref{prop:pgla-stabilizer}. Composing $J^A$ with the $j$-invariant gives a morphism $f^A\colon V(n)^A\to\bA^1_k$ such that for every $j\in k\smallsetminus\{0,1728\}$, the fiber $V_j^A\coloneqq (f^A)^{-1}(j)$ parameterizing twisted elliptic normal curves whose Jacobian has $j$-invariant $j$ is a homogeneous space under the simply connected semisimple $k$-group $\SL(A)$ with finite solvable geometric stabilizer. The morphism $f^A$ is smooth away from $0$ and $1728$, while the fibers of $f^A$ at $0$ and $1728$ are nowhere-reduced. 

This geometric construction brings the problem into the range of the Brauer--Manin obstruction machinery. We argue as follows. By a well-known deformation theory argument, Artin's $n$-gon determines by descent a smooth $k$-point of the closure of $V(n)^A$ in the Hilbert scheme of $\SB(A)$; see Proposition~\ref{nodal}. Therefore, by Hironaka's resolution of singularities, we can find a smooth projective $k$-variety $W$ containing $V(n)^A$ as a dense open subset, together with a morphism $\bar f\colon W\to \bP^1_k$ extending $f^A$. Then
\begin{itemize}
	\item[(i)] since $A$ is split over $k(\SB(A))$, the morphism $\Bar{f}$ has a rational section over $k(\SB(A))$ and hence, $W$ being proper and $k$ being algebraically closed in $k(\SB(A))$, all the fibers of $\Bar{f}$ are split (that is, they contain a non-empty geometrically integral open subscheme);
	\item[(ii)] the generic fiber of $\Bar{f}$ is geometrically integral and, since $A$ is split over $\kbar$, it admits a smooth $\kbar(t)$-point;
	\item[(iii)] because the geometric generic fiber of $\Bar{f}$ is rationally connected, its Picard group is torsion-free and its Brauer group is finite;
	\item[(iv)] there exists a nonempty open subscheme $U\subset\bA^1_k$ such that for all $u\in U(k)$, the fiber $W_u$ is a smooth compactification of $V_u$, and $V_u$ is an $\SL(A)$-homogeneous space with \emph{solvable} finite geometric stabilizer, so that by a theorem of Demeio \cite[Theorem 1.1]{demeio2026solvable} (see Theorem~\ref{demeio-thm} below), $W_u(k)$ is dense in the (unramified) Brauer--Manin set of $W_u$ with respect to the product of the $v$-adic topologies (as defined in \textsection\ref{subsec:brauer-manin}).
\end{itemize}

In (i), note that although the fibers of $f^A$ at $0$ and $1728$ are nowhere-reduced and hence are not split, the corresponding fibers of the compactified fibration $\bar f$ contain split irreducible components. We now apply the fibration method: by a theorem of Harari \cite[Proposition 3.1.1]{harari1997fleches} (see Theorem~\ref{harari-thm} below), properties (i)-(iv) imply that $W(k)$ is dense in the Brauer--Manin set of $W$ with respect to the product of the $v$-adic topologies. Since $W(k)$ is not empty, the Brauer--Manin set of $W$ is not empty. Moreover, as $W_{\kbar}$ is rationally connected, the cokernel of $\Br(k)\to \Br(W)$ is finite, and hence a simple approximation argument at finitely many places of $k$ (see Lemma~\ref{lemma:bm-open-rc}) suffices to show that $V(n)^A(k)$ is not empty, as desired.

Our proof of Theorems~\ref{maincor} and \ref{mainthm} are contained in Section~\ref{sec:thm1.2}, while Section~\ref{sec:hilb} contains background on the Hilbert scheme of twisted elliptic normal curves, the geometry of its fibers over the $j$-line, and a review of folklore results about degenerations of twisted elliptic normal curves inside cyclic Severi--Brauer varieties.  Our proof of Theorem~\ref{local-global-thm} is contained in Section~\ref{sec:local-global} and our proof of Theorem~\ref{weak-approximation-thm} is contained in Section~\ref{section-fixed-j-inv}.

\paragraph{Notation.} 
Let $k$ be a field, let $\kbar$ be a separable closure of $k$, let $\Gamma_k = \Gal(\kbar/k)$ be the absolute Galois group of $k$, and let $\Br(k)$ be the Brauer group of $k$.
        
By definition, a $k$-variety is a separated $k$-scheme of finite type. For a smooth proper $k$-variety $V$, denote by $\Pic_{V/k}$ the Picard scheme of $V$, by $\Pic^0_{V/k}$ the connected component of $\Pic_{V/k}$ containing the identity element, and by $\Pic(V)$ the Picard group of $V$. When $V$ is of pure dimension $1$, the abelian variety $\Pic^0_{V/k}$ is the Jacobian of $V$. Let $\Br(V)$ be the Brauer group of $V$, let $\Br_1(V)$ be the kernel of the base-change map $\Br(V)\to \Br(V_{\kbar})$, and let $\Br_0(V)$ be the image of the pullback map $\Br(k)\rightarrow\Br(V)$. For an arbitrary integral $k$-variety $V$, write $\Br_{\nr}(V)\subset \Br(V)$ for the unramified Brauer group of $V$: when $k$ has characteristic zero and $V$ is smooth this subgroup coincides with the image of $\Br(V^c)\to \Br(V)$, where $V^c$ is any smooth compactification of $V$. Let $\Br_{1,\nr}(V)\coloneqq \Br_{\nr}(V)\cap\Br_1(V)$.
		
If $S$ is a $k$-scheme and $s\in S$, we denote by $k(s)$ the residue field of $s$. If $f\colon X\to S$ and $T\to S$ are morphisms of schemes, we write $X_T$ for $X\times_ST$. In particular, if $s\in S$, we write $X_s=X\times_S\Spec (k(s))$. For a locally free sheaf $\VV$ on $S$, we use the quotient convention $\mathbf P(\VV)\coloneqq \operatorname{Proj}_S(\operatorname{Sym}\VV)$, so that $\mathbf P(\VV)$ parametrizes rank-one quotients of $\VV$. Thus a surjection $f^*\VV\twoheadrightarrow \LL$ onto an invertible sheaf $\LL$ defines a morphism $X\to \bP(\VV)$.
		
If $G$ is a profinite group and $M$ is a finite discrete $G$-module, let $M^G$ be the subgroup of $G$-invariant elements of $M$, and for fixed $g\in G$ let $M^g$ be the subgroup of $g$-invariant elements of $M$. For all $i\geq 0$, denote by $H^i(G,M)$ the $i$th continuous cohomology group of $M$, and for fixed $g \in G$ we let $g^*\colon H^i(G,M)\to H^i(\bZ,M)$ be the pullback map along the group homomorphism $\bZ\to G$ sending $1\in\bZ$ to $g$, where $1\in \bZ$ acts on $M$ via $g$. If $G=\Gamma_k$ is the absolute Galois group of $k$, we set $\cdual{M}\coloneqq \mathrm{Hom}_{\bZ}(M,\mu_{\infty})$, where $\mu_{\infty}\subset \kbar^\times$ is the $\Gamma_k$-submodule of roots of unity in $\kbar$.

For a free $\bZ/n\bZ$-module of finite rank $M$, we let $\GL(M)\coloneqq \Aut_{\bZ/n\bZ}(M)$ and $\SL(M)\subset \GL(M)$ be the kernel of the determinant $\det\colon \GL(M) \to (\bZ/n\bZ)^\times$. A choice of a $\bZ/n\bZ$-basis of $M$ identifies these groups with $\GL_r(\bZ/n\bZ)$ and $\SL_r(\bZ/n\bZ)$, where $r$ is the rank of $M$, respectively.
		
If $H$ is an algebraic $k$-group (for example, a finite constant group), we let $Z(H)$, $[H,H]$, and $H^{\mathrm{ab}}$ be the center, the derived subgroup, and the abelianization of $H$, respectively. If $H$ is a finite constant group, we also let $H/\mathrm{conj}$ be the set of conjugacy classes of $H$.  As a standard abuse of notation, if $G$ is a group scheme over a scheme $T$ and $X \to T$ is a $T$-scheme, we often also denote by $G$ the pullback of $G$ to $X$.
		
\paragraph{Acknowledgments.}
We thank Danny Krashen, Eoin Mackall, and David Saltman for conversations related to the topic of the present paper. Thanks also go to Julian Demeio for sharing his results on the Brauer--Manin obstruction with us before they were released, and to Olivier Wittenberg for suggesting the use of Harari's Theorem~\ref{harari-thm} to simplify our original proof of Theorem~\ref{mainthm}. The first author was supported by NSF grants DMS-2102010 and DMS-2152235, by Simons Fellowships 666565 and 00005925, and by the Simons Collaboration on Perfection (Award MP-SCMPS-00001529-11).  The second author was supported by NSF grant DMS-2200845 and by a Simons Foundation Collaboration Grant 712097.

\section{Hilbert schemes of twisted elliptic normal curves}
\label{sec:hilb}

Our proofs of Theorems~\ref{maincor} and~\ref{mainthm} are based on the following notion.

\begin{definition}
Let $k$ be a field, let $n\geq 3$, and let $X$ be a Severi--Brauer variety of dimension $n-1$ over $k$. A \emph{twisted elliptic normal curve} in $X$ is a smooth proper geometrically connected $k$-curve $C\subset X$ of genus $1$ such that for some (equivalently, every) field extension $K/k$ splitting $X$ and for some (equivalently, every) $K$-isomorphism $X_K\cong \bP^{n-1}_K$, the composite $C_K\hookrightarrow X_K\cong \bP^{n-1}_K$ is defined by a complete linear system on $C_K$, or equivalently it has degree $n$ and its image is not contained in any hyperplane.
\end{definition} 

In particular, a twisted elliptic normal curve in $\bP^{n-1}_k$ is a smooth projective geometrically connected $k$-curve $C\subset \bP^{n-1}_k$ of genus $1$, degree $n$, and not contained in any hyperplane.

\subsection{Preliminaries on the moduli stack of elliptic curves}

Let $k$ be a field, and let $\sM_{1,1}$ be the stack of elliptic curves over $k$. By \cite[Theorem 13.1.2]{olsson2016algebraic}, the stack $\sM_{1,1}$ is a smooth separated Deligne--Mumford stack over $k$. We let 
\begin{equation}\label{eq:universal-elliptic}
    \pi\colon \sE\longrightarrow \sM_{1,1},\qquad \Sigma\subset \sE
\end{equation}
be the universal elliptic curve and the image of its zero section, respectively.

By \cite[Theorem 13.1.15]{olsson2016algebraic}, the coarse moduli space morphism \begin{equation}\label{eq:j-invariant}\jj\colon \sM_{1,1}\longrightarrow \bA^1_k
\end{equation}
sends an elliptic curve $E$ to its $j$-invariant $\jj(E)$. For every elliptic curve $E$ over $k$, let $\mathrm{Aut}_{E/k}$ be the automorphism group $k$-scheme of $E$ as an elliptic curve; it is a finite \'etale $k$-group. 

For an algebraic $k$-stack $\mathcal X$, we let $\mathcal X_{\mathrm{red}}$ be the reduction of $\mathcal X$ in the sense of \cite[Definition 050C]{stacks}. The following lemma is standard.

\begin{lemma}\label{lem:m11}
Let $k$ be a field such that $\mathrm{char}(k)\notin\{2,3\}$.
\begin{enumerate}
\item[{\em (1)}] Let $U\coloneqq \bA^1_k\smallsetminus\{0,1728\}$. The morphism $\jj$ restricts to a neutral $\mu_2$-gerbe over $U$:
\[
\jj^{-1}(U)\cong U\times_k B\mu_2.
\]
\item[{\em (2)}] We have
\[
\jj^{-1}(0)_{\red}\cong B\mu_6,
\qquad
\jj^{-1}(1728)_{\red}\cong B\mu_4 .
\]
\item[{\em (3)}] Let $K/k$ be a field extension, and let $E$ be an elliptic curve over $K$. If $\jj(E)\notin\{0,1728\}$, then $\jj^{-1}(\jj(E))\cong B\Aut_{E/K}$. If $\jj(E)\in\{0,1728\}$, then $\jj^{-1}(\jj(E))_{\red}\cong B\Aut_{E/K}$.
\end{enumerate}
\end{lemma}

\begin{proof}
We use the presentation
\[
\sM_{1,1}\cong [X/\Gm],
\qquad
X=\Spec k[a,b,\Delta^{-1}],
\qquad
\Delta=-16(4a^3+27b^2),
\]
where $\lambda\in\Gm$ acts by $\lambda\cdot(a,b)=(\lambda^4a,\lambda^6b)$. 
A point of $[X/\Gm]$ with coordinates $(a,b)$ corresponds to the elliptic curve of affine equation $y^2=x^3+ax+b$. The map $\jj$ corresponds to a morphism
\[
[X/\Gm]\longrightarrow \bA^1_k,\qquad (a,b)\longmapsto 1728\frac{4a^3}{4a^3+27b^2}.
\]

For (1), let $X_U\coloneqq X\times_{\mathbf A^1_k}U$. Observe that $a,b$ are invertible on $X_U$. We have a $\Gm$-equivariant isomorphism
\[
U\times_k \Gm \xlongrightarrow{\sim} X_U,\qquad
(j,t)\longmapsto
\left(
\frac{3jt^2}{1728-j},
\frac{2jt^3}{1728-j}
\right)
\]
where $\Gm$ acts on $U\times_k\Gm$ by $\lambda\cdot(j,t)=(j,\lambda^2t)$. The inverse of this isomorphism is 
\[
X_U \xlongrightarrow{\sim} U\times_k \Gm,\qquad (a,b)\longmapsto
\left(
1728\frac{4a^3}{4a^3+27b^2},
\frac{3b}{2a}
\right).
\]
Therefore $\jj^{-1}(U)\cong [(U\times_k \Gm)/\Gm]$, where $\Gm$ acts trivially on $U$ and with weight $2$ on $\Gm$, and hence $\jj^{-1}(U)\cong U\times_k B\mu_2$.

For (2), at $j=0$, the scheme-theoretic fiber is cut out by $a^3=0$. Its reduction is given by $a=0$, with $b\neq0$, and therefore $\jj^{-1}(0)_{\red}
\cong [\Gm/\Gm]$, where $\Gm$ acts with weight $6$ on itself. Hence $\jj^{-1}(0)_{\red}\cong B\mu_6$. The proof of the isomorphism  $\jj^{-1}(1728)_{\red}\cong B\mu_4$ is entirely analogous.

Finally, (3) is an immediate consequence of (1) and (2).
\end{proof}

\subsection{The Hilbert scheme of degree \texorpdfstring{$n$}{n} twisted elliptic normal curves}
		
Let $k$ be a field, and let $n\geq 3$ be an integer. We let $V(n)$ be the Hilbert scheme of twisted elliptic normal curves in $\bP^{n-1}_k$, which is defined as follows.  See \cite[Section~2]{mackall:enc} for a similar construction. Let $P(t)=nt\in \bZ[t]$, and let $\Hilb(n)$ be the Hilbert scheme parametrizing closed subschemes of $\bP^{n-1}_k$ with Hilbert polynomial $P(t)$. Thus, for every $k$-scheme $S$, we have a functorial identification
		\begin{align}\label{hilbert-points}
			\begin{gathered}
				\Hilb(n)(S)=\{\text{$S$-flat closed subschemes $X\subset \bP^{n-1}_S$ such that $X_s\subset \bP^{n-1}_{k(s)}$ has dimension $1$,} \\\text{degree $n$, and arithmetic genus $1$ for every $s\in S$}\}.
			\end{gathered}
		\end{align}
		Let $\sC\subset \bP^{n-1}_k\times_k\Hilb(n)$ be the universal family, and let $V(n)\subset\Hilb(n)$ be the open subscheme consisting of points $z\in \Hilb(n)$ such that the fiber $\sC_z$ is smooth, not contained in any hyperplane of $\bP^{n-1}_{k(z)}$, and geometrically connected. Thus, $V(n)$ is a quasi-projective $k$-scheme of finite type and, for every $k$-scheme $S$, the identification of (\ref{hilbert-points}) induces a functorial identification
		\begin{align}\label{vn-points}
			\begin{gathered}
				V(n)(S)=\{\text{$S$-smooth closed subschemes $X\subset \bP^{n-1}_S$ such that $X_s\subset \bP^{n-1}_{k(s)}$}\\ \text{is a twisted elliptic normal curve for every $s\in S$}\}.
			\end{gathered}
		\end{align}
		By definition, a smooth genus $1$ curve over a $k$-scheme $S$ is a smooth proper morphism $X\to S$ whose geometric fibers are connected curves of geometric genus $1$. For every $k$-scheme $S$, consider the groupoid of triples $(f\colon X\to S, \LL, \varphi)$, where $f$ is a smooth genus $1$ curve over $S$, $\LL$ is an invertible sheaf of degree $n$ on $X$, and $\varphi\colon \bP(f_*\LL)\to \bP^{n-1}_S$ is an isomorphism. By definition, an isomorphism from a triple $(f\colon X\to S,\LL,\varphi)$ to a triple $(f'\colon X'\to S,\LL',\varphi')$ is a pair $(\alpha,\beta)$, where $\alpha\colon X\xrightarrow{\sim}X'$ is an isomorphism of $S$-schemes and $\beta\colon\LL\xrightarrow{\sim}\alpha^*\LL'$ is an isomorphism of invertible sheaves, such that we have a commutative triangle of isomorphisms
\[
\begin{tikzcd}
\bP(f_*\LL) \arrow[rr,"\sim"] \arrow[dr,"\varphi"'] &&
\bP(f'_*\LL') \arrow[dl,"\varphi'"] \\
& \bP^{n-1}_S
\end{tikzcd}
\]
where the top horizontal map is induced by $\beta$.

\begin{lemma}\label{lem:vn-points-2}
For every $k$-scheme $S$ we have a functorial identification
\begin{align*}
V(n)(S)=\{\text{triples $(f\colon X\to S, \LL, \varphi)$, where $f$ is a smooth genus $1$ curve, $\LL$ is an invertible sheaf}\\ \text{of degree $n$ on $X$, and $\varphi\colon \bP(f_*\LL)\xrightarrow{\sim} \bP^{n-1}_S$ is an isomorphism over $S$}\}/\mathrm{isom}.\nonumber
\end{align*}
\end{lemma}

\begin{proof}
    We let $\sF_n$ be the contravariant functor from $k$-schemes to sets which to a $k$-scheme $S$ associates the collection of isomorphism classes of triples $(f\colon X\to S, \LL, \varphi)$ as above.

    Let $X\subset \mathbf P^{n-1}_S$ be an $S$-point of $V(n)$, let
    $f\colon X\to S$ be the structure morphism, and set $\LL=\OO_{\mathbf P^{n-1}_S}(1)|_X$. For all $s\in S$, since every geometric
    fiber of $f$ is a genus $1$ curve and $\LL|_{X_s}$ has degree $n\geq 3$, we have $H^1(X_s,\LL_s)=0$. Hence $f_*\LL$ is locally free of rank $n$, and its formation commutes with base change. Moreover, for every $s\in S$, the restriction map $H^0(\mathbf P^{n-1}_{k(s)},\OO(1))\to H^0(X_s,\LL_s)$ is injective because $X_s$ is not contained in a hyperplane. Since both
    sides have dimension $n$, this map is an isomorphism. It follows from the
    base change theorem that the natural map $\OO_S^{\oplus n}\to f_*\LL$ induced by the coordinate functions on $\mathbf P^{n-1}_S$ is an isomorphism, and hence  determines an isomorphism $\varphi\colon \mathbf P(f_*\LL)\xrightarrow{\sim}
        \mathbf P^{n-1}_S$. This defines a natural transformation $\Phi\colon V(n)\to \sF_n$.

    Conversely, let $(f\colon X\to S,\LL,\varphi)\in \sF_n(S)$.
    Since $\LL_s$ has degree $n\geq 3$ on the genus $1$ curve $X_s$,
    it is very ample and satisfies $H^1(X_s,\LL_s)=0$ for every
    geometric point $s$ of $S$. Hence $f_*\LL$ is locally free of rank~$n$, its formation commutes with base change, and the relative complete
    linear system defines a morphism $X\to \mathbf P(f_*\LL)$. This morphism is a closed immersion because it is so on all fibers.
    Composing it with $\varphi$ gives an $S$-flat closed subscheme
    $X\subset \mathbf P^{n-1}_S$ whose fibers are smooth genus $1$
    curves of degree $n$ not contained in any hyperplane. Therefore $X\subset \mathbf P^{n-1}_S$ defines an $S$-point of $V(n)$. This gives a natural transformation $\Psi\colon \sF_n\to V(n)$.

    We leave it to the reader to check that $\Phi$ and $\Psi$ are
    quasi-inverse to each other.
\end{proof}

		 Taking the Jacobian of the universal curve over $V(n)$ yields a morphism 
        \begin{equation}\label{eq:J}J\colon V(n)\longrightarrow\sM_{1,1}.
        \end{equation}
        Under the identification of Lemma~\ref{lem:vn-points-2}, if an $S$-point of $V(n)$ is
represented by a triple $(f\colon X\to S,\LL,\varphi)$, then its image
under $J$ is isomorphic to the elliptic curve $\Pic^0_{X/S}\to S$. The $k$-group $\PGL_n$ acts on $\Hilb(n)$ and $V(n)$ via its action on $\bP^{n-1}_k$, and $J$ is $\PGL_n$-invariant. Since $n\geq 3$, the invertible sheaf $\OO(n\Sigma)$  on $\sE$ (cf. \eqref{eq:universal-elliptic}) is relatively very ample, and hence it defines a closed immersion over $\sM_{1,1}$
\begin{equation}\label{eq:iota_n}
\iota_n\colon
\sE\lhook\joinrel\longrightarrow
\bP_{\sE,n},\qquad \bP_{\sE,n}\coloneqq \bP(\pi_*\OO(n\Sigma)).
\end{equation}
For every morphism $T\to \sM_{1,1}$ and every
$P\in \sE[n](T)$, translation by $P$ on $\sE_T$ preserves
$\OO(n\Sigma)_T$ up to isomorphism. Hence it induces a well-defined
projective automorphism of $(\bP_{\sE,n})_T$. Thus we obtain a homomorphism of relative group schemes over $\sM_{1,1}$
\begin{equation}\label{eq:h_n}
h_n\colon
\sE[n]\longrightarrow
\Aut_{\sM_{1,1}}(\bP_{\sE,n}).
\end{equation}
By construction, $\iota_n$ is equivariant with respect to $h_n$, that is, we have a commutative  diagram
\begin{equation}\label{eq:iota-h}
\begin{tikzcd}
\sE[n]\times_{\sM_{1,1}}\sE
    \arrow[r,"\tau"]
    \arrow[d,"h_n\times \iota_n"'] &
\sE
    \arrow[d,"\iota_n"] \\
\Aut_{\sM_{1,1}}(\bP_{\sE,n})
\times_{\sM_{1,1}}
\bP_{\sE,n}
    \arrow[r] &
\bP_{\sE,n},
\end{tikzcd}
\end{equation}
where $\tau$ is the translation action of $\sE[n]$ on $\sE$ and where the bottom arrow is the tautological action. 

\begin{lemma}\label{lem:h_n-closed}
    Let $k$ be a field, and let $n\geq 3$ be an integer. The morphism $h_n$ of \eqref{eq:h_n} is a closed immersion.
\end{lemma}

\begin{proof}
It is enough to show that $h_n$ is
a proper monomorphism. The morphism $h_n$ is proper because
$\sE[n]\to \sM_{1,1}$ is finite and $\Aut_{\sM_{1,1}}(\bP_{\sE,n})\to \sM_{1,1}$ is separated. To show that $h_n$ is a monomorphism, it suffices to show that
its kernel is trivial. Let $T$ be a $k$-scheme, let $T\to\sM_{1,1}$ be a morphism, and let
$P\in\sE[n](T)$ lie in the kernel of $h_n$. Then the projective
automorphism $h_n(P)$ is the identity. By \eqref{eq:iota-h}, translation
by $P$ acts trivially on $\sE_T$. Evaluating at the zero section gives
$P=0$. Thus $h_n$ is a monomorphism.
\end{proof}

Set
\[
\sR_n
\coloneqq
\operatorname{Isom}_{\sM_{1,1}}(\bP_{\sE,n},
\mathbf P^{n-1}_{\sM_{1,1}}
)\xlongrightarrow{q}\sM_{1,1}.
\]
Then $q$ is a bitorsor for $\operatorname{Aut}_{\sM_{1,1}}(\bP_{\sE,n})$ and $\PGL_n$: the first group acts on the right by precomposition, and
$\PGL_n$ acts on the left by postcomposition. Via $h_n$, this gives a right
action of $\sE[n]$ on $\sR_n$, and the left action of $\PGL_n$
commutes with it. 

For a morphism $T\to\sM_{1,1}$ and an isomorphism $r\colon
(\bP_{\sE,n})_T
\xrightarrow{\sim}
\mathbf P^{n-1}_T$, define $\Theta(r)$ to be the embedded curve $r\circ \iota_{n,T}\colon \sE_T
\hookrightarrow
(\bP_{\sE,n})_T
\xlongrightarrow{\sim}
\mathbf P^{n-1}_T$. 
This gives a family of twisted elliptic normal curves over $T$, i.e., a $T$-point of $V(n)$, and its
Jacobian is canonically $\sE_T\to T$. 
We obtain a morphism of $\sM_{1,1}$-schemes 
\[
\Theta\colon \sR_n\longrightarrow V(n)
\]
i.e., fitting into a commutative diagram
\begin{equation}\label{eq:rn-vn-m11}
\begin{tikzcd}
\sR_n \arrow[rr,"\Theta"] \arrow[dr,"q"'] & & V(n) \arrow[dl,"J"] \\
& \sM_{1,1}. &
\end{tikzcd}
\end{equation}

The following is a crucial geometric observation.

\begin{proposition}\label{prop:pgl-homogeneous}
Let $k$ be a field, and let $n\geq 3$ be an integer. The morphism $\Theta$ is equivariant for the left action of $\PGL_n$ and is a fppf torsor for the right action of $J^*\sE[n]$ on $\sR_n$. In
particular, if $n$ is invertible in~$k$, then $\Theta$ is étale, $J$ is
smooth, and the $k$-variety $V(n)$ is smooth and geometrically integral.
\end{proposition}

\begin{proof}
We first show that $\Theta$ is invariant under the right action of $J^* \sE[n]$, or equivalently, by the commutativity of \eqref{eq:rn-vn-m11}, of
$\sE[n]$. Let $T\to\sM_{1,1}$ be a morphism, let
$P\in\sE[n](T)$, and let $r\colon
(\bP_{\sE,n})_T
\xrightarrow{\sim}
\mathbf P^{n-1}_T$
be a $T$-point of $\sR_n$. By \eqref{eq:iota-h}, we have 
\[
h_n(P)\circ \iota_{n,T}
=
\iota_{n,T}\circ \tau_P,
\]
where we write $\tau_P$ for translation by $P$. Thus we have
\[
\Theta(r \circ h_n(P)) = r \circ h_n(P) \circ \iota_{n,T} = r \circ \iota_{n,T} \circ \tau_P = \Theta(r) \circ \tau_P
\]
so the embeddings associated to $r$ and to $r\circ h_n(P)$ differ by the
automorphism $\tau_P$ of $\sE_T$. They therefore define the same point
of $V(n)$. Thus $\Theta$ is $\sE[n]$-invariant.

It remains to prove that $\Theta$ is an fppf torsor under $J^*\sE[n]$.
Let $T\to V(n)$ be a morphism, corresponding to a family $X\subset \mathbf P^{n-1}_T$ of twisted elliptic normal curves, let $E\coloneqq \Pic^0_{X/T}$, and let $p\colon E\to T$ be the structure morphism.
Observe that $J^*\sE[n]$ pulls back to $E[n]$ over $T$, and that the zero section $\Sigma$ of $\sE$ pulls back to the zero section $\Sigma_E$ of $E$. We show that the $E[n]$-action on
\[
\sR_n\times_{V(n)}T\longrightarrow T
\]
makes it into an fppf torsor under $E[n]$. This assertion is fppf-local on $T$. After an fppf base change, we may choose
a section of $p \colon X\to T$, hence identify $X$ with its Jacobian $E$. Under such an
identification, $\OO_{\mathbf P^{n-1}_T}(1)|_X$ becomes an invertible sheaf of relative degree $n$ on $E$. Since the $T$-morphism
\begin{align*}
E &{} \longrightarrow \Pic^n_{E/T} \\
P &{} \longmapsto \tau_P^*\OO(n\Sigma_E)
\end{align*}
is an fppf torsor under $E[n]$, fppf-locally on $T$ there exist
$P\in E(T)$, an invertible sheaf $M$ on $T$ and an isomorphism
\[
\OO_{\mathbf P^{n-1}_T}(1)|_X
\cong
\tau_P^*\OO(n\Sigma_E)\otimes p^*M.
\]
After composing the
chosen identification $E\cong X$ with $\tau_{-P}$, we may assume that $P=0$, that is,
\[
\OO_{\mathbf P^{n-1}_T}(1)|_X
\cong\OO(n\Sigma_E)\otimes p^*M
\]
for some invertible sheaf $M$ on $T$. The associated complete linear systems then give an isomorphism
\[
\mathbf P(p_*\OO(n\Sigma_E))
\xlongrightarrow{\sim}
\mathbf P^{n-1}_T,
\]
and hence a lift of the given $T$-point of $V(n)$ to $\sR_n$. Therefore
$\sR_n\times_{V(n)}T\to T$ has sections fppf-locally on $T$.

Now suppose that two lifts over $T$ are given. In particular, they give two isomorphisms $\alpha,\beta\colon E\xrightarrow{\sim}X$ compatible with the identification of $E$ with the Jacobian of $X$, and such
that
\begin{equation}\label{eq:alpha-beta-isom}
\alpha^*\OO_{\mathbf P^{n-1}_T}(1)|_X
\cong
\OO(n\Sigma_E)\otimes p^*M,
\qquad
\beta^*\OO_{\mathbf P^{n-1}_T}(1)|_X
\cong
\OO(n\Sigma_E)\otimes p^*M'
\end{equation}
for some invertible sheaves $M,M'$ on $T$. Since $\alpha$ and $\beta$ induce
the same identification on Jacobians, there is a \emph{unique} section
$P\in E(T)$ such that $\beta=\alpha\circ\tau_P$. The two isomorphisms of \eqref{eq:alpha-beta-isom} then give $\tau_P^*\OO(n\Sigma_E)
\cong
\OO(n\Sigma_E)\otimes p^*N$ 
for some invertible sheaf $N$ on $T$. Equivalently, $P$ stabilizes the class of
$\OO(n\Sigma_E)$ in $\Pic^n_{E/T}$. The stabilizer of this class is
$E[n]$: indeed, under the identification $E\cong \Pic^0_{E/T}$ determined by the origin of $E$,  the difference
$\tau_P^*\OO(n\Sigma_E)\otimes\OO(n\Sigma_E)^{-1}$
corresponds to $nP$. Thus $P\in E[n](T)$. Thus the $E[n](T)$-action on the set of lifts is simply transitive whenever the latter is non-empty. In conclusion, we have proved that for every $T\to V(n)$ the pullback $\sR_n\times_{V(n)}T\to T$ is locally non-empty for the fppf topology and that $E[n]=(J^*\sE[n])|_T$ acts simply
transitively on its sections. Hence it is an fppf torsor under $E[n]$.

If $n$ is invertible in $k$, then $\sE[n]\to\sM_{1,1}$ is finite
étale, hence $\Theta$ is étale. Since $q\colon\sR_n\to\sM_{1,1}$ is a
$\PGL_n$-torsor, it is smooth. The identity $q=J\circ\Theta$, together with
the fact that $\Theta$ is étale and surjective, implies that $J$ is smooth.
\end{proof}


\subsection{The homogeneous spaces \texorpdfstring{$V_E^A$}{VEA} and \texorpdfstring{$V_j^A$}{VjA}}\label{twisted-case}
		
Let $k$ be a field, let $n\geq 3$ be an integer invertible in $k$, let $A$ be a central simple algebra of degree $n$ over $k$. 
We write $\GL(A)$ for the smooth affine $k$-group defined by $\GL(A)(R)=(A\otimes_k R)^\times$ for every $k$-algebra $R$. The reduced norm defines a morphism of algebraic
groups $\operatorname{Nrd}\colon \GL(A)\to \Gm$. We set
\[
\SL(A)\coloneqq \ker(\operatorname{Nrd}),\qquad \PGL(A)\coloneqq \GL(A)/\Gm,
\]
where $\Gm$ is embedded in $\GL(A)$ by scalar multiplication. We have a commutative diagram of $k$-groups with exact rows
\begin{equation}\label{eq:mun-sla-pgla}
\begin{tikzcd}
    1\arrow[r] & \mu_n \arrow[r]\arrow[d,hook]  &\SL(A) \arrow[r] \arrow[d,hook]  & \PGL(A)\arrow[r] \arrow[d,equal]  & 1 \\
    1\arrow[r] & \Gm \arrow[r] &\GL(A) \arrow[r] & \PGL(A)\arrow[r] & 1.
\end{tikzcd}
\end{equation}
        
Let $P^A$ be the right $\PGL_n$-torsor over $k$ corresponding to $A$. For a left quasi-projective $\PGL_n$-scheme $Y$ over $k$, we denote by $Y^A$ the twist of $Y$ by $P^A$, that is, 
\begin{equation}\label{eq:twist-by-pgl-torsor}
Y^A\coloneqq P^A\times^{\PGL_n}Y=(P^A\times_k Y)/\PGL_n,
\end{equation}
where $\PGL_n$ acts diagonally. In particular, $(\bP^{n-1}_k)^A=\SB(A)$ is the Severi--Brauer variety of $A$, and $(\PGL_n)^A=\PGL(A)$ is the automorphism $k$-group of $P^A$ and of $\SB(A)$.



Twisting the description in Lemma~\ref{lem:vn-points-2}, we see that $V(n)^A$ is isomorphic to the Hilbert scheme of twisted elliptic normal curves of degree $n$ in the Severi--Brauer variety $\SB(A)$, that is, for every $k$-scheme $S$, we have a functorial identification
\begin{align}\label{vna-points}
\begin{gathered}
V(n)^A(S)=\{\text{$S$-smooth closed subschemes $X\subset \SB(A)_S$ such that $X_s\subset \SB(A\otimes_kk(s))$}\\ \text{is a twisted elliptic normal curve for every $s\in S$}\}.
\end{gathered}
\end{align}
By Proposition~\ref{prop:pgl-homogeneous}, the $k$-variety $V(n)^A$ is smooth and geometrically integral.

Twisting \eqref{eq:rn-vn-m11} by $P^A$ yields a commutative triangle
\begin{equation}\label{eq:rn-vn-m11-a}
\begin{tikzcd}
\sR_n^A \arrow[rr,"\Theta^A"] \arrow[dr,"q^A"'] & & V(n)^A \arrow[dl,"J^A"] \\
& \sM_{1,1} &
\end{tikzcd}
\end{equation}
where $\sR_n^A\coloneqq (\sR_n)^A= \operatorname{Isom}_{\sM_{1,1}}(\bP_{\sE,n}, \SB(A)_{\sM_{1,1}})$. We define the composite morphism
		\begin{equation}\label{eq:morphism-f}f^A\colon  V(n)^A\xlongrightarrow{J^A} \sM_{1,1} \xlongrightarrow{\jj} \bA^1_k,
        \end{equation}
where $\jj$ was defined in \eqref{eq:j-invariant}.

\begin{proposition}\label{prop:pgl-homogeneous-a}
Let $k$ be a field, let $n\geq 3$ be an integer, and let $A$ be a central simple algebra of degree $n$ over $k$. The morphism $\Theta^A$ is equivariant for the left action of $\PGL(A)$ and is a torsor for the right action of $(J^A)^*\sE[n]$ on $\sR_n^A$. In particular, if $n$ is invertible in $k$, then $\Theta^A$ is étale and $J^A$ is
smooth.
\end{proposition}

\begin{proof}
    This follows from Proposition~\ref{prop:pgl-homogeneous} by twisting by $P^A$. 
\end{proof}

 Suppose that $n\geq 3$ is invertible in $k$ and that $\mathrm{char}(k)\notin\{2,3\}$. By Proposition~\ref{prop:pgl-homogeneous-a} and Lemma~\ref{lem:m11}, $f^A$ is smooth over $\bA^1_k\smallsetminus\{0,1728\}$. Fix a field extension $K/k$, an element $j\in K$, and an elliptic curve $E$ over $K$ such that $\jj(E)=j$. We define the $K$-schemes $V_E^A$ and $V_j^A$ as the fiber products
        \[
        \begin{tikzcd}
        V_E^A \arrow[r] \arrow[d]  & V(n)^A \arrow[d,"J^A"] \\
        \mathrm{Spec}(K) \arrow[r,"E"] & \sM_{1,1},
        \end{tikzcd}
        \qquad\qquad 
        \begin{tikzcd}
        V_j^A\arrow[d]  \arrow[r]  & V(n)^A \arrow[d,"f^A"] \\
        \mathrm{Spec}(K) \arrow[r,"j"]  & \mathbf{A}^1_k
        \end{tikzcd}
        \]
        if $j\notin \{0,1728\}$, while if $j\in \{0,1728\}$ we define $V_E^A$ and $V_j^A$ as the maximal reduced closed subschemes of the fiber products above (cf. Lemma~\ref{lem:m11}). When $A$ is split, we simply write $V_E\coloneqq V_E^A$ and $V_j\coloneqq V_j^A$. Therefore, $V_j^A=(V_j)^A$ (resp. $V_E^A=(V_E)^A$) is the twist of $V_j$ by the $\PGL(A)$-torsor $P^A$. We have a commutative square
\begin{equation}\label{bigdiagram-twisted}
\begin{tikzcd}
V^A_E \arrow[r] \arrow[d]
  & V^A_{j} \arrow[r] \arrow[d]
  & V(n)^A \arrow[d,"J^A"] \arrow[dd,bend left=45,"f^A"] \\
\Spec(K) \arrow[r,"E"] \arrow[dr,equal]
  & B\Aut_{E/K} \arrow[r] \arrow[d]
  & \sM_{1,1} \arrow[d,"\jj"] \\
& \Spec(K) \arrow[r,"j"]
  & \bA^1_k .
\end{tikzcd}
\end{equation}
where the two squares on the top are Cartesian by definition, while the lower square is Cartesian if and only if $j\notin\{0,1728\}$. Observe also that the morphism $V_E^A\rightarrow V_j^A$ is an $\Aut_{E}$-torsor.

By \eqref{vna-points}, for every field extension $L/K$ we have the natural identifications
\begin{align}\label{vea-points}
			V_E^A(L)=\{\text{Pairs $(C,\psi)$, where $C\subset \SB(A_L)$ is a twisted elliptic normal curve}\\ \text{and $\psi\colon \Pic_{C/L}^0\xrightarrow{\sim} E_L$ is an $L$-group isomorphism}\},\nonumber
		\end{align}
        \begin{equation}\label{vja-points}
			V^A_j(L)=\{\text{Twisted elliptic normal curves $C\subset\SB(A_L)$ such that $\jj(\Pic^0_{C/L})=j$}\}.
		\end{equation}

Let $0_E\in E(K)$ be the origin, and let $\bP_{E,n}\coloneqq \bP(H^0(E,\OO(n\cdot 0_E)))$. Pulling back \eqref{eq:rn-vn-m11-a} along the morphism $\Spec(K)\to \sM_{1,1}$ determined by $E$ gives a $K$-morphism
\begin{equation}\label{eq:p-mod-en}
    \Theta_E^A\colon R_E^A\longrightarrow V_E^A,
\end{equation}
where $R_E^A=\mathrm{Isom}_K(\bP_{E,n},\SB(A)_K)$. A choice of basis of the $n$-dimensional $K$-vector space $H^0(E,\OO(n\cdot 0_E))$ determines an isomorphism $\bP_{E,n}\cong \bP_K^{n-1}$ and hence compatible isomorphisms $\mathrm{Aut}(\bP_{E,n})\cong \PGL_{n,K}$ and $R_E^A\cong (P^A)_K$.

\begin{proposition}\label{prop:pgla-stabilizer}
    Let $k$ be a field, let $n\geq 3$ be an integer invertible in $k$, and let $A$ be a central simple algebra of degree $n$ over $k$. Let $K/k$ be a field extension, let $E$ be an elliptic curve over $K$, let $0_E\in E(K)$ be the origin, and let $\bP_{E,n}\coloneqq \bP(H^0(E,\OO(n\cdot 0_E)))$. The morphism $\Theta_E^A$ of \eqref{eq:p-mod-en} is equivariant for the left action of $\PGL(A)_K$ and a torsor for the right action of $E[n]$ on $R_E^A=\operatorname{Isom}_K(\bP_{E,n},\SB(A)_K)$ induced by the inclusion $h_n\colon E[n]\hookrightarrow\mathrm{Aut}(\bP_{E,n})$ of \eqref{eq:h_n}.  In particular,
            \[V_E^A\cong R_E^A/E[n].\]
    Furthermore, if $A_K$ is split, then $V_E^A(K)\neq\emptyset$.
\end{proposition}

\begin{proof}
    The first part is an immediate consequence of Proposition~\ref{prop:pgl-homogeneous-a}. Furthermore, if $A_K$ is split, then $\SB(A)_K\cong \bP^{n-1}_K\cong \bP_{E,n}$, so that $R_E^A(K)\neq\emptyset$ and hence $V_E^A(K)\neq\emptyset$.
\end{proof}   

For the proofs of Theorems~\ref{local-global-thm} and \ref{weak-approximation-thm}, it will be necessary to study in detail the geometric stabilizer of $V_E^A$ viewed as a $\SL(A)$-homogeneous space; we will do this in \S\ref{subsec:geometric-stabilizers-detail}. For the proof of Theorems~\ref{maincor} and \ref{mainthm}, the following consequence of Proposition~\ref{prop:pgla-stabilizer} will suffice.
		
\begin{corollary}\label{solvable-stabilizers}
Let $k$ be a field of characteristic not equal to $2$ or $3$, let $n\geq 3$ be an integer invertible in~$k$, and let $A$ be a central simple algebra of degree $n$ over $k$.  Fix $j\in k$, and let $E$ be an elliptic curve over $k$ such that $\jj(E)=j$. Then $V_E^A$ and $V_{j}^A$ have the structure of $\SL(A)$-homogeneous spaces whose geometric stabilizers are solvable finite groups.
\end{corollary}

\begin{proof}
    The finite \'etale $k$-group $\Aut_{E/k}$ is solvable (and even cyclic, unless $j=0$ and $\mathrm{char}(k)\in\{ 2,3\}$). The conclusion follows from  Proposition~\ref{prop:pgla-stabilizer} and the fact that the map $V_E^A\to V_j^A$ appearing in \eqref{bigdiagram-twisted} is an $\Aut_{E/k}$-torsor.
\end{proof}

\subsection{Degenerations of twisted elliptic normal curves inside cyclic Severi--Brauer varieties}
		
The following proposition is folklore (apparently going back to Michael Artin) and was known to
experts; see for example Jason Starr's answer to~\cite{starr_g1c_mo} or \cite[Proof of Proposition 3.2, Remark 3.3]{mackall:enc}.
		
\begin{proposition}\label{nodal}
Let $k$ be a field and let $A$ be a cyclic central simple algebra of degree $n\geq 3$ over $k$. Then the closure of $V(n)^A$ in $\Hilb(n)^A$ contains a smooth $k$-point.
\end{proposition}
		
\begin{proof}
Let $x_0,\dots,x_{n-1}$ be $n$ projectively independent $k$-points of $\bP^{n-1}_k$, and let $C\subset \bP^{n-1}_k$ be the union of the $n$ projective lines $\mathrm{Span}(x_i,x_j)$, where $j\equiv i+1\pmod n$. Let $\nu\colon \widetilde{C}\to C$ be the normalization map, so that $\widetilde{C}$ is isomorphic to the disjoint union of $n$ copies of $\bP^1_k$.
			
\begin{claim}\label{claim-unobstructed}
We have $H^1(C,\OO_C(1))=0$.
\end{claim}

\begin{proof}[Proof of Claim~\ref{claim-unobstructed}]
The $k$-points $x_0,\ldots,x_{n-1}$ are the nodes of $C$. For each $0\leq m\leq n-1$, let
$\epsilon_m\colon \Spec k\hookrightarrow C$ be the inclusion of $x_m$ in $C$. The
normalization sequence is
\[
0\longrightarrow \OO_C\longrightarrow \nu_*\OO_{\widetilde C}
\longrightarrow \bigoplus_{m=0}^{n-1}(\epsilon_m)_*k\longrightarrow 0,
\]
where the last map records the difference of the two values at the two
preimages of each node. Since $\nu^*\OO_C(1)=\OO_{\widetilde C}(1)$,
tensoring with $\OO_C(1)$ gives a short exact sequence
\[
0\longrightarrow \OO_C(1)\longrightarrow
\nu_*\OO_{\widetilde C}(1)
\longrightarrow \bigoplus_{m=0}^{n-1}(\epsilon_m)_*k\longrightarrow 0.
\]
Thus $\chi(\OO_C(1))=\chi(\OO_{\widetilde C}(1))-n=2n-n=n$. We now compute $h^0(C,\OO_C(1))$. A section of
$\OO_{\widetilde C}(1)$ is an $n$-tuple of sections, one on each component
of $\widetilde C\cong\coprod_{m=0}^{n-1}\mathbf P^1_k$. Hence $h^0(\widetilde C,\OO_{\widetilde C}(1))=2n$. Such an $n$-tuple descends to $C$ if and only if the two values over each node $x_m$ agree. These are $n$ independent linear conditions. Therefore $h^0(C,\OO_C(1))=2n-n=n$. Since $\chi(\OO_C(1)) = h^0(C,\OO_C(1)) =n$, it follows that $h^1(C,\OO_C(1))=0$. 
\end{proof}

			\begin{claim}\label{claim-hilbert-smooth}
				The curve $C$ defines a smooth $k$-point $[C]$ of the Hilbert scheme $\Hilb(n)$ of (\ref{hilbert-points}), and every open neighborhood of $[C]$ intersects $V(n)$ non-trivially.
			\end{claim}
			
			\begin{proof}[Proof of Claim~\ref{claim-hilbert-smooth}]
				Note that $C$ has degree $n$. By the genus formula for nodal curves, $C$ has arithmetic genus $1$. Thus $C$ has Hilbert polynomial $P(t)=nt$, and thus defines a $k$-point of $\Hilb(n)$. The conclusion follows from Claim~\ref{claim-unobstructed} and the deformation theory of the Hilbert scheme, due to Grothendieck \cite{grothendieck1962fondements}. More precisely, by \cite[Proposition~29.9]{hartshorne2010deformation}, the curve $C$ is smoothable inside $\bP^{n-1}_k$, that is, the unique irreducible component of $\Hilb(n)$ containing $[C]$ admits a dense open subscheme that parametrizes  elliptic normal curves. Moreover, Claim~\ref{claim-unobstructed} implies that $H^1(C,\NN)=0$, where $\NN$ is the normal sheaf of $C\subset \bP^{n-1}_k$ (see the proof of \cite[Proposition 29.9]{hartshorne2010deformation}), and by \cite[Theorem 1.1(c)]{hartshorne2010deformation} this in turn implies that the Hilbert scheme is smooth at $[C]$.
			\end{proof}
			
			Since $A$ is cyclic, there exists a Galois $G$-algebra $L/k$ splitting $A$, where $G$ is a cyclic group of order $n$. We fix an isomorphism $\SB(A)\cong \bP^{n-1}_L/G$ over $k$, where $G$ acts on $\bP^{n-1}_L$ by cyclically permuting the coordinates up to a scalar; see \cite[Construction 2.5.1]{gille2017central}. Let $C'\subset \SB(A)$ be the closed subscheme corresponding to $C_L/G$ under this isomorphism. Then $C'$ defines a $k$-point $[C']\in \Hilb(n)^A(k)$. By Claim~\ref{claim-hilbert-smooth}, the $k$-scheme $\Hilb(n)^A$ is smooth at $[C']$, and every open neighborhood of $[C']$ intersects $V(n)^A$ non-trivially. 
		\end{proof}
		
		\begin{remark}\label{rmk:large}    
			Recall that a field $k$ is said to be \emph{large} if every irreducible $k$-curve that has a smooth $k$-point has infinitely many $k$-points, see \cite{pop:large}. It follows that $X(k)$ is dense in $X$ for every irreducible $k$-variety $X$ that has a smooth $k$-point. 
			Thus Proposition~\ref{nodal} implies that, if $k$ is a large field and  $A$ is a cyclic algebra of degree $n\geq 3$ over $k$, then $V(n)^A(k)$ is Zariski-dense in $V(n)^A$.	
		\end{remark}

\section{Proof of Theorem~\ref{mainthm}}
\label{sec:thm1.2}

In this section, we study the Brauer--Manin obstruction for the existence of rational points on the Hilbert scheme of twisted elliptic normal curves, leading up to a proof of Theorem~\ref{mainthm}.
        
\subsection{Preliminaries on the Brauer--Manin obstruction for homogeneous spaces}\label{subsec:brauer-manin}

We follow the notation of \cite[\textsection 2.1 and \textsection 2.2]{demeio2026solvable}. Let $k$ be a number field, let $\Omega_k$ be the set of places of $k$, and for every $v\in \Omega_k$ let $k_v$ be the completion at $v$. Let $V$ be a smooth geometrically integral $k$-scheme of finite type, and consider the topological space \[V(k_\Omega)\coloneqq \prod_{v\in\Omega_k} V(k_v),\] which is endowed with the product of the $v$-adic topologies on each $V(k_v)$. For every $v\in\Omega_k$, let $\operatorname{inv}_v\colon \Br(k_v)\to \bQ/\bZ$
be the local invariant map. The (unramified) \emph{Brauer--Manin pairing} is
\begin{align}\label{eq:Brauer--Manin-pairing}
V(k_\Omega)\times \Br_{\nr}(V) &{} \longrightarrow \bQ/\bZ \\
(\{x_v\}_{v\in\Omega_k},\beta) &{} \longmapsto
        \sum_{v\in\Omega_k}\operatorname{inv}_v(\beta(x_v)),
\end{align}
where $\Br_{\nr}(V)$ is the unramified Brauer group of $V$. Observe that since $\beta$ is unramified, we have $\beta(x_v)=0$ for all but finitely many $v$, and hence the pairing is well defined. By definition, the (unramified) \emph{Brauer--Manin set} of $V$ is the subset of $V(k_\Omega)$ orthogonal to $\Br_{\nr}(V)$ under the Brauer--Manin pairing. We always consider this set with the product of the $v$-adic topologies.  By the Albert--Brauer--Hasse--Noether theorem, we have an exact sequence
\begin{equation}\label{eq:albert-brauer-hasse-noether}
        0\longrightarrow \Br(k)\longrightarrow
        \bigoplus_{v\in\Omega_k}\Br(k_v)
        \xlongrightarrow{\sum_v\operatorname{inv}_v}
        \bQ/\bZ
        \longrightarrow 0.
\end{equation}
It follows that the diagonal image of $V(k)$ in $V(k_\Omega)$ is contained in the Brauer--Manin set of $V$.

Our proofs of Theorems~\ref{mainthm} and~\ref{maincor} rely on the following.

\begin{theorem}[Demeio]\label{demeio-thm}
Let $k$ be a number field, and let $V$ be a homogeneous space of a
semisimple simply connected linear algebraic group over $k$, with finite
solvable geometric stabilizers. Then $V(k)$ is dense in the Brauer--Manin set of $V$. In particular, if the Brauer--Manin set of $V$ is non-empty, then $V(k)$ is non-empty.
\end{theorem}
\begin{proof}
See \cite[Theorem~1.1]{demeio2026solvable}
\end{proof}

We will also use the following observation.

\begin{lemma}\label{lemma:bm-open-rc}
Let $k$ be a number field, let $X$ be a smooth proper geometrically
integral $k$-scheme, and let $U\subset X$ be a dense open subscheme.
Assume that $\Br(X)/\Br(k)$ is finite and that $X(k)$ is dense
in the Brauer--Manin set of $X$. If $X(k)$ is non-empty, then $U(k)$ is non-empty.
\end{lemma}

\begin{proof}
Fix $\alpha_1,\dots,\alpha_r\in \Br(X)$
whose images in $\Br(X)/\Br(k)$ generate. Let $P\in X(k)$. For every $\alpha\in \Br(X)$, the evaluation
map $X(k_v)\to\Br(k_v)$ given by $x\mapsto \alpha(x)$ is locally
constant. Therefore, for each place $v$, there is an open neighbourhood
$O_v\subset X(k_v)$ of $P$, in the $v$-adic topology, such that
$\alpha_i(x)=\alpha_i(P)$ for all $x\in O_v$ and all $i=1,\dots,r$.
Since $X$ is smooth, $U(k_v)$ is dense in $X(k_v)$ for the $v$-adic topology, and hence we may choose
$x_v\in U(k_v)\cap O_v$, for every $v\in \Omega_k$. Thus, for every
$i=1,\dots,r$ we have $\alpha_i(x_v)=\alpha_i(P)$ in $\Br(k_v)$. We show that $\{x_v\}_{v\in \Omega_k}$ belongs to the
Brauer--Manin set of $U$. For any $\beta\in \Br_{\nr}(U)=\Br(X)$, we may write $\beta=\sum_{i=1}^r n_i\alpha_i+\gamma$ for some $n_i\in\bZ$ and some $\gamma\in\Br(k)$. Then, for every
$v\in\Omega_k$,
\[
        \beta(x_v)
        =
        \sum_{i=1}^r n_i\alpha_i(x_v)+\gamma_v
        =
        \sum_{i=1}^r n_i\alpha_i(P)+\gamma_v
        =
        \beta(P)\mbox{ in $\Br(k_v)$,}
\]
where $\gamma_v$ denotes the image of $\gamma$ in $\Br(k_v)$. Therefore
\[
        \sum_{v\in\Omega_k}\inv_v(\beta(x_v))
        =
        \sum_{v\in\Omega_k}\inv_v(\beta(P))=0 \mbox{ in $\bQ/\bZ$,}
\]
where the second equality follows from \eqref{eq:albert-brauer-hasse-noether}.
Therefore $\{x_v\}_{v\in \Omega_k}$ lies in the Brauer--Manin set of $U$, as desired.

Fix a place $v_0\in \Omega_k$. Since $U$ is Zariski-open in $X$, the subset $U(k_{v_0})$ is an
open neighborhood of $x_{v_0}\in X(k_{v_0})$. Since $X(k)$ is dense in the Brauer--Manin set of $X$, there exists $Q\in X(k)$ whose image in $X(k_{v_0})$ lies in $U(k_{v_0})$. Since $U$ is Zariski-open in $X$, this implies that $Q\in U(k)$. Therefore $U(k)$ is non-empty, as desired.
\end{proof}

\begin{remark}\label{rem:RC}
We remark that by \cite[Theorem~5.5.2]{colliot2021brauer}, if $X_{\kbar}$ is rationally connected then the quotient $\Br(X)/\Br(k)$ is finite.
\end{remark}

\subsection{Preliminaries on the fibration method}

The fibration method is a circle of techniques to prove that the Brauer--Manin obstruction for weak approximation is the only one for varieties that admit a fibration structure over $\bP^1$ with rationally connected fibers.  See \cite[Section~1]{harpaz-wittenberg-fibration} for a survey.
        
A separated scheme of finite type over a perfect field $k$ is said to be \emph{split} if it contains a non-empty geometrically integral open subscheme. This definition is due to Skorobogatov; see \cite[Definition 0.1]{skorobogatov1996descent} or \cite[Definition 9.1.3]{colliot2021brauer}. 
		
We shall use the following form of the fibration method in the proofs of Theorems~\ref{maincor} and~\ref{mainthm}.

		\begin{theorem}[Harari]\label{harari-thm}
			Let $k$ be a number field, let $V$ be a separated geometrically integral $k$-scheme of finite type, let $f\colon V\to \bA^1_k$ be a surjective morphism, let $k(t)$ be the function field of $\bA^1_k$, and let $X/k(t)$ be a smooth projective compactification of the generic fiber of $f$. Assume:
			\begin{itemize}
				\item[{\em (i)}] for every closed point $P\in \bA^1_k$, the fiber $V_P$ is a split $k(P)$-scheme;
				\item[{\em (ii)}] the generic fiber $V_{k(t)}$ of $f$ is geometrically integral over $k(t)$ and admits a smooth $\kbar(t)$-point;
				\item[{\em (iii)}] $\Pic(X_{\overline{k(t)}})$ is torsion-free and $\Br(X_{\overline{k(t)}})$ is finite;
				\item[{\em (iv)}] there exists a non-empty open subscheme $U\subset\bA^1_k$ such that for all $u\in U(k)$, if $(V_u)^c$ is a smooth projective compactification of $V_u$, then $(V_u)^c(k)$ is dense in the Brauer--Manin set of $(V_u)^c$.
			\end{itemize}
			Then, for every smooth projective compactification $V^c$ of $V$, the set $V^c(k)$ is dense in the Brauer--Manin set of $V^c$ with respect to the product of the $v$-adic topologies.
		\end{theorem}
\begin{proof}
This is \cite[Proposition 3.1.1]{harari1997fleches}, which is a variant of \cite[Theorem 4.2.1]{harari1994methode}.
\end{proof}
We shall verify assumption (i) of Theorem~\ref{harari-thm} using the following criterion.

		\begin{lemma}\label{skorobogatov-criterion}
			Let $k$ be a perfect field, let $X\to \bP^1_k$ be a morphism, where $X$ is a smooth proper irreducible $k$-scheme. Let $K/k$ be a field extension such that $k$ is algebraically closed in $K$ and $X_K\to \bP^1_K$ has a section. Then, for every closed point $P$ of $\bP^1_k$, the fiber $X_P$ is a split $k(P)$-scheme.
		\end{lemma}
		
		\begin{proof}
            This is a special case of \cite[Proposition 10.1.8]{colliot2021brauer}.
		\end{proof}

		\subsection{Proofs of Theorems~\ref{maincor} and~\ref{mainthm}}

		\begin{proof}[Proof of Theorem~\ref{mainthm}]
			Let $A$ be a central simple $k$-algebra such that $X=\mathbf{SB}(A)$, let $n=\dim(X)+1\geq 3$ be the degree of $A$, and let $P^A$ be the $\mathbf{PGL}_n$-torsor over $k$ corresponding to $A$. Let $f^A\colon V(n)^A\to \bA^1_k$ be the $\mathbf{PGL}(A)$-invariant morphism of \eqref{eq:morphism-f}. Since $k$ is a number field, it follows from the Albert--Brauer--Hasse--Noether theorem that $A$ is cyclic. Therefore, by Proposition~\ref{nodal}, there exists a smooth $k$-variety $V'$ containing $V(n)^A$ as a dense open subscheme and such that $V'(k)\neq\emptyset$. Since $k$ has characteristic zero, by Hironaka's resolution of singularities, there exist a smooth projective geometrically integral $k$-variety $W$ containing $V'$ (and hence $V(n)^A$) and a morphism $\Bar{f}^A\colon W\to \bP^1_k$ whose restriction to $V(n)^A$ is $f^A$. Since $V'$ has a $k$-point, so does $W$.
			
			We now check that the assumptions of Theorem~\ref{harari-thm} are satisfied by the restriction of $\Bar{f}^A$ over $\bA^1_k\subset \bP^1_k$. Let $t$ be the standard coordinate on $\bA^1_k$.
			
			\begin{itemize} 
				\item[(i)] Let $K\coloneqq k(\SB(A))$. Then $k$ is algebraically closed in $K$. Moreover, by Amitsur's theorem \cite{amitsur1955generic} (see also \cite[Theorem 5.4.1]{gille2017central}) the central simple algebra $A_K$ is split, and hence by Proposition~\ref{prop:pgla-stabilizer} (applied to any elliptic curve $E$ over $k(t)$ such that $\jj(E)=t$) the generic fiber of $(f^A)_K$ has a $K(t)$-point, that is, a rational section. By the valuative criterion for properness, any morphism from the generic point of a curve extends to the smooth proper model, which implies that $(\Bar{f}^A)_K$ has a section. Lemma~\ref{skorobogatov-criterion} now implies that, for every closed point $P\in \bA^1_k$, the fiber of $\Bar{f}^A$ at $P$ is a split $k(P)$-scheme.
				\item[(ii)] The generic fiber of $\Bar{f}^A$ is a smooth projective compactification of the generic fiber of $f^A$, which by \eqref{bigdiagram-twisted} is isomorphic to $V_t^{A_{k(t)}}$ over $k(t)$. Proposition~\ref{prop:pgla-stabilizer} (applied to any elliptic curve $E$ over $k(t)$ such that $\jj(E)=t$) implies that $V_t^{A_{k(t)}}$ is smooth over $k(t)$ and, since $A_{\kbar}$ is split, that $V_t^{A_{k(t)}}(\kbar(t))\neq \emptyset$. Thus, the generic fiber of $\Bar{f}^A$ is geometrically integral and has a smooth $\kbar(t)$-point.
				\item[(iii)] The geometric generic fiber of $\Bar{f}^A$ is a compactification of the $\SL(A_{\overline{k(t)}})$-homogeneous space $V_{\overline{k(t)}}$. Therefore it is unirational, and so in particular its Picard group is torsion-free and its Brauer group is finite by \cite[Theorem 5.5.2 and (6) p. 347]{colliot2021brauer}.
                
				\item[(iv)] If $U=\bA^1_k \smallsetminus\{0,1728\}$ and $u\in U(k)$, then by Corollary~\ref{solvable-stabilizers} the fiber $(f^A)^{-1}(u)=V_u^A$ is a homogeneous space under $\SL(A)$ with finite solvable stabilizer. By Theorem~\ref{demeio-thm}, $V_u^A(k)$ is dense in the Brauer--Manin set of $V_u^A$. A fortiori, if $(V_u^A)^c$ is a smooth projective compactification of $V_u^A$, then $(V_u^A)^c(k)$ is dense in the Brauer--Manin set of $(V_u^A)^c$ with respect to the product of the $v$-adic topologies.
			\end{itemize}
			
			By Theorem~\ref{harari-thm}, we deduce that $W(k)$ is dense in the Brauer--Manin set of $W$. 
			Since $W(k)\neq\emptyset$, the Brauer--Manin set of $W$ is not empty. Since $W_{\kbar}$ is rationally connected (being the total space of a fibration over $\bP^1_{\kbar}$ with rationally connected geometric generic fiber), by Lemma~\ref{lemma:bm-open-rc} and Remark~\ref{rem:RC} we deduce that $V(n)^A(k)\neq\emptyset$, as desired.
		\end{proof}

		\begin{proof}[Proof of Theorem~\ref{maincor}]
			Let $\alpha\in \Br(k)$. There exists a central simple algebra $A$ of degree $n\geqslant 3$ such that $\alpha=[A]$ in $\Br(k)$. By Theorem~\ref{mainthm}, there exists a twisted elliptic normal curve $C\subset \SB(A)$ over $k$. In particular, $\alpha_{k(C)}=0$ in $\Br(k(C))$. 
		\end{proof}

\section{Proof of Theorem~\ref{local-global-thm}}
\label{sec:local-global}

        In this section, we study the Brauer--Manin obstruction to the local-global principle for the Hilbert scheme of twisted elliptic normal curves, leading up to a proof of Theorem~\ref{local-global-thm}.
        
		\subsection{Preliminaries on the unramified Brauer group of homogeneous spaces}
		
		Let $k$ be a field of characteristic zero, let $G$ be an algebraic $k$-group, let $V$ be a homogeneous space under $G$ over $k$, let $\Bar{v}\in V(\kbar)$ be a geometric point of $V$, and let $H_{\Bar{v}}$ be the $G(\kbar)$-stabilizer of $\Bar{v}$. We suppose that $H_{\Bar{v}}$ is finite.  Let $G_{\Bar{v}} \subset G(\kbar) \rtimes \Gamma_k$ be the subgroup consisting of those pairs $(g, \sigma)$ such that $g\sigma(\Bar{v}) = \Bar{v}$. We have a short exact sequence of abstract groups
		\begin{equation}\label{eq:outer-action}
		1 \longrightarrow H_{\Bar{v}} \longrightarrow G_{\Bar{v}} \longrightarrow \Gamma_k \longrightarrow 1.
		\end{equation}
		This short exact sequence induces an outer action of $\Gamma_k$ on $H_{\Bar{v}}$, that is, a continuous group homomorphism $\Gamma_k\to \mathrm{Out}(H_{\Bar{v}})$. This outer action induces an action of $\Gamma_k$ on the abelianization $H^{\mathrm{ab}}_{\Bar{v}}$. More explicitly, let $\sigma \in \Gamma_k$, and choose a lift $(g_\sigma,\sigma)$ of $\sigma$ to
$G_{\Bar{v}}$, that is, $g_\sigma \sigma(\Bar{v})=\Bar{v}$. Then the outer automorphism of $H_{\Bar{v}}$ induced by $\sigma$ is given by $h\mapsto g_\sigma \sigma(h) g_\sigma^{-1}$, and the induced automorphism of $H_{\Bar{v}}^{\mathrm{ab}}$ is independent of the choice of $g_\sigma$.

        Recall that, for a finite $\Gamma_k$-module $M$, we write $\cdual{M}$ for the $\Gamma_k$-module $\mathrm{Hom}_{\bZ}(M,\mu_{\infty})$, where $\mu_{\infty}\subset \kbar^\times$ is the $\Gamma_k$-module of roots of unity. If $M$ has exponent $e$ and $k$ contains a primitive $e$th root of unity, then $\cdual{M}\cong \mathrm{Hom}_{\bZ}(M,\bQ/\bZ)\cong \mathrm{Hom}_{\bZ}(M,\bZ/e\bZ)$, where $\Gamma_k$ acts trivially on $\bQ/\bZ$ and $\bZ/e\bZ$.

		\begin{lemma}\label{h-ab}
			Let $k$ be a field of characteristic zero, let $G$ be a simply connected semisimple $k$-group, let $V$ be a homogeneous space under $G$ over $k$, let $\Bar{v}\in V(\kbar)$ be a geometric point of $V$, and let $H_{\Bar{v}}$ be the $G(\kbar)$-stabilizer of $\Bar{v}$. We suppose that $H_{\Bar{v}}$ is finite. 
            \begin{enumerate}
            \item[{\em (1)}] We have $\kbar[V]^{\times}=\kbar^\times$.
            
            \item[{\em (2)}] We have an isomorphism of $\Gamma_k$-modules $\Pic(V_{\kbar})\cong \cdual{H}^{\mathrm{ab}}_{\Bar{v}}$.
            \end{enumerate}
		\end{lemma}
		
		\begin{proof}
        (1) Since $G$ is semisimple, by Rosenlicht's Lemma we have $\kbar[G]^{\times}=\kbar^\times$; see for example \cite[Corollary p. 369]{popov1972stability}. Since we have a surjective morphism $G_{\kbar}\to V_{\kbar}$ over $\kbar$, this implies $\kbar[V]^{\times}=\kbar^\times$. 
            
		(2)	See for example \cite[(5.2)]{harpaz2020zero}.
		\end{proof}
		
		 In view of Lemma~\ref{h-ab}, the Hochschild–Serre spectral sequence 
        \[E^{p,q}_2\coloneqq H^p(k,H^q(V_{\kbar},\Gm))\Longrightarrow H^{p+q}(V,\Gm)\]
        yields an exact sequence
		\begin{equation}\label{hochschild-serre-seq}
			\Br(k) \longrightarrow \Br_1(V) \longrightarrow H^1(k, \cdual{H}^{\mathrm{ab}}_{\Bar{v}}) \xlongrightarrow{\delta} H^3(k, \kbar^\times).
		\end{equation}
		The map $\delta$ is trivial if $V$ admits a zero-cycle of degree $1$, by the functoriality of the Hochschild--Serre spectral sequence and a restriction-corestriction argument, or if $k$ is a number field, since then $H^3(k,\kbar^\times)=0$; see \cite[Remark 3.1]{harpaz-wittenberg-massey}. When $\delta=0$, we obtain an isomorphism $\Br_1(V)/\Br_0(V)\cong H^1(k,\cdual{H}_{\Bar{v}}^{\mathrm{ab}})$.

		\begin{proposition}\label{brauer-homogeneous}
			Let:
			\begin{itemize}
				\item[--] $k$ be a field of characteristic zero,
				\item[--] $G$ be a simply connected semisimple $k$-group,
				\item[--] $V$ be a homogeneous space under $G$ over $k$,
				\item[--] $\Bar{v}\in V(\kbar)$ be a geometric point of $V$,
				\item[--] $H_{\Bar{v}}\subset G(\kbar)$ be the stabilizer of $\Bar{v}$, of exponent $e\geq 1$, such that $k$ contains a primitive $e$th root of unity,
				\item[--] $\delta\colon H^1(k, \cdual{H}^{\mathrm{ab}}_{\Bar{v}})\to H^3(k, \kbar^\times)$ be the map appearing in (\ref{hochschild-serre-seq}), 
				\item[--] $L/k$ be a finite Galois subextension of $\kbar/k$ such that $V(L)\neq\emptyset$ and such that the outer action of $\Gamma_k$ on $H_{\Bar{v}}$ factors through $\Gal(L/k)$, and
				\item[--] $\mathrm{Inf}\colon H^1(\Gal(L/k),\cdual{H}^{\mathrm{ab}}_{\Bar{v}})\to H^1(k,\cdual{H}^{\mathrm{ab}}_{\Bar{v}})$ be the (injective) inflation map.
			\end{itemize} 		
			Then the image of $\Br_{1,\nr}(V)$ in $H^1(k, \cdual{H}^{\mathrm{ab}}_{\Bar{v}})$ coincides with the subgroup of the classes of the form $\mathrm{Inf}(\beta)$, where $\beta \in H^1(\Gal(L/k), \cdual{H}^{\mathrm{ab}}_{\Bar{v}})$ is such that $\delta(\mathrm{Inf}(\beta)) = 0$ and for every $g \in \Gal(L/k)$, letting $\bZ$ act on $H^{\mathrm{ab}}_{\Bar{v}}$ and on $\cdual{H}^{\mathrm{ab}}_{\Bar{v}}$ through $g$, the pullback $g^*\beta \in H^1(\bZ, \cdual{H}^{\mathrm{ab}}_{\Bar{v}})$ is orthogonal, with respect to the cup product pairing \begin{equation}\label{cup-product-pairing}
				H^0(\bZ, H^{\mathrm{ab}}_{\Bar{v}}) \times H^1(\bZ, \cdual{H}^{\mathrm{ab}}_{\Bar{v}}) \longrightarrow H^1(\bZ, \bQ/\bZ) = \bQ/\bZ,
			\end{equation} to the image of the natural map $(H_{\Bar{v}}/\mathrm{conj})^g \rightarrow (H^{\mathrm{ab}}_{\Bar{v}})^g = H^0(\bZ, H^{\mathrm{ab}}_{\Bar{v}})$.
		\end{proposition}
		
		\begin{proof}
			This is a special case of a result of Harpaz and Wittenberg \cite[Corollary 3.4]{harpaz-wittenberg-massey}, which builds upon previous work of Harari, Demarche, and Lucchini-Arteche.
		\end{proof}
		
		\begin{corollary}\label{brauer-corollary}
			Under the assumptions of Proposition~\ref{brauer-homogeneous}, let $\beta \in H^1(\Gal(L/k), \cdual{H}^{\mathrm{ab}}_{\Bar{v}})$ be such that  $\mathrm{Inf}(\beta)$ is in the image of the map $\Br_{1,\nr}(V)\to H^1(k, \cdual{H}^{\mathrm{ab}}_{\Bar{v}})$, and let $g\in \Gal(L/k)$ be such that the image of the natural map $(H_{\Bar{v}}/\mathrm{conj})^g \rightarrow (H^{\mathrm{ab}}_{\Bar{v}})^g$ generates $(H^{\mathrm{ab}}_{\Bar{v}})^g$. Then $g^*\beta=0$ in $H^1(\bZ,\cdual{H}^{\mathrm{ab}}_{\Bar{v}})$.
		\end{corollary}

		\begin{proof}
			Let $M$ be a finite abelian group of exponent $e$ with a $\bZ$-action. Let $\sigma$ be the automorphism of $M$ determined by the action of $1\in \bZ$. The tautological perfect pairing $M\times \cdual{M}\to \bQ/\bZ$ induces a perfect pairing $M^\sigma\times \cdual{M}/(\sigma-1)\to \bQ/\bZ$, which is isomorphic to the cup product pairing $H^0(\bZ,M)\times H^1(\bZ,\cdual{M})\to \bQ/\bZ$. Therefore the cup product pairing (\ref{cup-product-pairing}) is perfect. 
			
			By Proposition~\ref{brauer-homogeneous}, the class $g^*\beta\in H^1(\bZ,\cdual{H}^{\mathrm{ab}}_{\Bar{v}})$  is orthogonal, under the cup product pairing (\ref{cup-product-pairing}), to the subgroup of $(H_{\Bar{v}}^{\mathrm{ab}})^g$ generated by the image of $(H_{\Bar{v}}/\mathrm{conj})^g \rightarrow (H^{\mathrm{ab}}_{\Bar{v}})^g$. By assumption, this subgroup is equal to $(H_{\Bar{v}}^{\mathrm{ab}})^g=H^0(\bZ,H^{\mathrm{ab}}_{\Bar{v}})$. Thus $g^*\beta$ is orthogonal to $H^0(\bZ,H_{\Bar{v}}^{\mathrm{ab}})$. Since the pairing is perfect, we conclude that $g^*\beta=0$, as desired.
		\end{proof}

\subsection{Geometric stabilizers of \texorpdfstring{$V_E^A$}{VEA}: group theory and outer Galois action}\label{subsec:geometric-stabilizers-detail}
		        
In this section, we include more refined information about the stabilizers of the homogeneous space $V_E^A$, which is needed for the proofs of Theorems~\ref{local-global-thm} and~\ref{weak-approximation-thm}.

\begin{definition}\label{heisenberg-def}
Let $k$ be a field, let $n\geq 3$ be an integer invertible in $k$, and let $e_1,\dots,e_n$ be a basis of the $k$-vector space $k^n$. The $k$-group $\bZ/n\bZ\times_k \mu_n$ acts projectively on $\bA^n_k$ as follows: for all $k$-algebras $R$ and all $(a,\zeta)\in (\bZ/n\bZ\times_k \mu_n)(R)$, we have $(a,\zeta)\cdot(e_i)=\zeta^i e_{i+a}$, where the indices are considered modulo $n$. This defines an injective homomorphism 
\[\kappa_n\colon \bZ/n\bZ\times_k \mu_n \lhook\joinrel\longrightarrow \PGL_n.\]
By definition, the \emph{Heisenberg group} $\mathrm{He}_n$ is the central extension of $\bZ/n\bZ\times_k \mu_n$ defined by the following commutative diagram with exact rows:
\[
\begin{tikzcd}
    1 \arrow[r] & \mu_n \arrow[d,equal] \arrow[r]  & \mathrm{He}_n \arrow[r] \arrow[d,hook] & \bZ/n\bZ\times_k \mu_n \arrow[d,hook,"\kappa_n"]\arrow[r] & 1 \\
    1 \arrow[r] & \mu_n \arrow[r] & \SL_n \arrow[r] & \PGL_n \arrow[r] & 1.
\end{tikzcd}
\]
\end{definition}

\begin{proposition}\label{picard-scheme-twisted}
Let $k$ be a field, let $n\geq 3$ be an integer invertible in $k$, and let $A$ be a central simple algebra of degree $n$ over $k$. Fix $j\in k$, and let $E$ be an elliptic curve over $k$ such that $\jj(E)=j$. For every $\Bar{v}\in V_E^A(\kbar)$, the $\SL(A)(\kbar)$-stabilizer of $\Bar{v}$ is isomorphic to $\mathrm{He}_n(\kbar)$.
\end{proposition}
		
\begin{proof}
We may assume that $k$ is algebraically closed, so that $A$ is split and hence $\SL(A)=\SL_n$. 
            
Let $\zeta_n\in k^\times$ be a primitive $n$-th root of unity, let $e_n\colon E[n](k)\times E[n](k)\to \mu_n(k)$ be the Weil pairing (see for example \cite[III.8]{silverman}), and let $P_1,P_2\in E[n](k)$ be such that $e_n(P_1,P_2)=\zeta_n$. By \cite[Lemma 3.4]{fisher2010pfaffian}, we may choose coordinates on $\bP^{n-1}_k$ such that $h_n(P_1)=\kappa_n(1,1)$ and $h_n(P_2)=\kappa_n(0,\zeta_n)$. It follows that $h_n(E[n](k))$ is conjugate to the image of $\kappa_n$. Therefore, by Proposition~\ref{prop:pgla-stabilizer}, the $\PGL_n(k)$-stabilizer of $\Bar{v}$ is conjugate to the image of $\kappa_n$, and hence the $\SL_n(k)$-stabilizer of $\Bar{v}$ is conjugate $\mathrm{He}_n(k)$. 
\end{proof}

\begin{lemma}\label{theta-lemma}
Let $k$ be an algebraically closed field of characteristic zero, let $n\geq 3$ be an integer, let $\mathrm{He}_n\subset \SL_n$ be as in Definition~\ref{heisenberg-def}, and write $H\coloneqq \mathrm{He}_n(k)$.
\begin{enumerate}
\item[{\em (a)}]	The group $H$ is solvable. Moreover, we have $Z(H)=[H,H]=\mu_n(k)\cong \bZ/n\bZ$ and $H^{\mathrm{ab}}\cong \bZ/n\bZ\times\mu_n(k)$. 

\item[{\em (b)}] The exponent of $H$ divides $2n$.

\item[{\em (c)}] Every automorphism of $H$ that induces the identity on $H^{\mathrm{ab}}$ is inner.

\item[{\em (d)}] The Bogomolov multiplier of $H$ is trivial. 

\item[{\em (e)}] Let $u_1,u_2\in H$ be such that their images $\Bar{u}_1,\Bar{u}_2\in H^{\mathrm{ab}}$ form a $(\bZ/n\bZ)$-basis of $H^{\mathrm{ab}}$, and let $\varphi\in\Aut(H)$ be an automorphism whose induced automorphism of $H^{\mathrm{ab}}$ fixes $\Bar{u}_1$. Then $\varphi(u_1)$ is conjugate to $u_1$.

\end{enumerate}
\end{lemma}
		
		\begin{proof}
			We may suppose that $k$ is algebraically closed. We fix a primitive $n$th root of unity $\zeta_n$. We write the standard basis of $k^n$ as $e_i$, where $i\in \bZ/n\bZ$. Let $S,D\in \GL_n(k)$ be the matrices defined by $S(e_i)=e_{i+1}$ and $D(e_i)=\zeta_n^i e_i$ for all $i\in\bZ/n\bZ$.
Then $S$ and $D$ map to $\kappa_n(1,1)$ and $\kappa_n(0,\zeta_n)$ in $\PGL_n(k)$. Since $k$ is algebraically closed, we may choose $\lambda,\mu\in k^\times$ such that $\det(\lambda S)=\det(\mu D)=1$, and we set
\[
        \sigma_1\coloneqq \lambda S,\qquad \sigma_2\coloneqq \mu D,\qquad \tau\coloneqq[\sigma_1,\sigma_2]=[S,D]=\zeta_n^{-1}\mathrm{I}_n.
\]
Then $\sigma_1,\sigma_2$ belong to $H$ and lift $\Bar{\sigma}_1\coloneqq \kappa_n(1,1)$ and $\Bar{\sigma}_2\coloneqq \kappa_n(0,\zeta_n)$, respectively, while $\tau$ generates the central subgroup $\mu_n(k)\subset \mathrm{He}_n(k)$. The elements $\sigma_1,\sigma_2$ generate $H$.
			
			(a) The fact that $H$ is solvable follows from the diagram appearing in Definition~\ref{heisenberg-def}. Since $\sigma_1$ and $\sigma_2$ generate $H$, the derived subgroup $[H,H]$ is the smallest normal subgroup containing their commutator $\tau$. Since $\tau$ is central, this implies that $[H,H]$ is generated by $\tau$. In particular, $H^{\mathrm{ab}}\cong \bZ/n\bZ\times \mu_n(k)$. The matrices in $\SL_n(k)$ that commute with $D$ are diagonal, and the only diagonal matrices in $\SL_n(k)$ that commute with $S$ are scalar. This shows that $Z(H)$ is generated by $\tau$.
			
			(b) Every element of $H$ can be written as $\sigma_1^a\sigma_2^b\tau^c$ for some integers $a,b,c$. Moreover, for all integers $a,a',b,b',c,c'$, we have
			\[(\sigma_1^a\sigma_2^b\tau^c)(\sigma_1^{a'}\sigma_2^{b'}\tau^{c'})=\sigma_1^{a+a'}\sigma_2^{b+b'}\tau^{c+c'-a'b}.\]
			By induction on $m\geq 1$, we get
			\[(\sigma_1^a\sigma_2^b\tau^c)^m=\sigma_1^{ma}\sigma_2^{mb}\tau^{mc-ab\frac{m(m-1)}{2}}\]
			for all integers $a,b,c$ and all $m\geq 1$. Now, if $h\in H$, there exist $a,b,c\in\bZ$ such that $h=\sigma_1^a\sigma_2^b\tau^c$. Since $\sigma_1^{2n}=\sigma_2^{2n}=1$, we conclude that
			\[h^{2n}=(\sigma_1^a\sigma_2^b\tau^c)^{2n}=\sigma_1^{2na}\sigma_2^{2nb}\tau^{n(2c-ab(2n-1))}=1\cdot 1\cdot 1=1.\]
			
			(c) Let $\varphi\in\Aut(H)$ be an automorphism which induces the identity on $H^{\mathrm{ab}}$. By (a), we have $\varphi(\sigma_1)=\sigma_1\tau^a$ and $\varphi(\sigma_2)=\sigma_2\tau^b$ for some $a,b\in\bZ$. Let $\psi\in\Aut(H)$ be the inner automorphism given by conjugation by $\sigma_1^b\sigma_2^{-a}$. Since $\sigma_1\sigma_2\sigma_1^{-1}=\sigma_2\tau$ and $\sigma_2\sigma_1\sigma_2^{-1}=\sigma_1\tau^{-1}$, we have 
            \[\psi(\sigma_1)=\sigma_1\tau^a=\varphi(\sigma_1),\qquad \psi(\sigma_2)=\sigma_2\tau^b=\varphi(\sigma_2).\] Thus $\psi=\varphi$ is inner.
			
			(d) More generally, consider a central short exact sequence of groups
			\[1\longrightarrow C\longrightarrow G\longrightarrow \Gamma\longrightarrow 1,\]
			where $C$ and $\Gamma$ are finite abelian groups. This sequence gives rise to a well-defined group homomorphism $\lambda_G\colon \bigwedge^2(\Gamma)\to C$, given by $\gamma_1\wedge\gamma_2\mapsto [g_1,g_2]$ for all $\gamma_1,\gamma_2\in\Gamma$, where $g_i\in G$ is a lift of $\gamma_i\in \Gamma$, for $i=1,2$. We let $S_G\coloneqq \ker(\lambda_G)$, and let $S_{\mathrm{bic}}$ be the subgroup of $S_G$ generated by the $\lambda_1\wedge\lambda_2$ which lie in $S_G$. By Bogomolov's formula \cite[Theorem 7.3]{colliot2007rationality}, the Bogomolov multiplier of $G$ is isomorphic to the dual of $S_G/S_{\mathrm{bic}}$. 
			
			We apply this to the central extension
			\[1\longrightarrow Z(H)\longrightarrow H\longrightarrow H^{\mathrm{ab}}\longrightarrow 1.\]
			Since $\bigwedge^2(H^{\mathrm{ab}})$ is a free $\bZ/n\bZ$-module generated by $\Bar{\sigma}_1\wedge \Bar{\sigma}_2$, every element of 	$\bigwedge^2(H^{\mathrm{ab}})$ is of the form $\lambda_1\wedge\lambda_2$ for some $\lambda_i\in H^{\mathrm{ab}}$. Thus $S_H=S_{\mathrm{bic}}$, and hence the Bogomolov multiplier of $H$ is trivial. In fact, since $[\sigma_1,\sigma_2]=\tau$ generates $Z(H)$, the map $\lambda_H$ is an isomorphism, and so $S_H=S_{\mathrm{bic}}=0$.
			
			(e) Let $u\coloneqq [u_1,u_2]$. Since $\Bar{u}_1$ and $\Bar{u}_2$ form a basis of $H^{\mathrm{ab}}$, the element $\Bar{u}_1\wedge \Bar{u}_2\in \bigwedge^2(H^{\mathrm{ab}})$ is a generator. As discussed in (d), the map $\lambda_H$ is an isomorphism. It follows that $u=\lambda_H(\Bar{u}_1\wedge \Bar{u}_2)$ generates $Z(H)$. In view of (a), we deduce that $\langle u \rangle=Z(H)=[H,H]$. 
			
			By assumption, $\varphi(u_1)$ and $u_1$ differ by an element of $[H,H]$. As $[H,H]=\langle u\rangle$, this implies that $\varphi(u_1)=u^iu_1$ for some integer $i$. From the identity $u_2^{-1}u_1u_2=u u_1$, we deduce that $\varphi(u_1)=u^i u_1=u_2^{-i}u_1u_2^i$, that is, $\varphi(u_1)$ is conjugate to $u_1$.
		\end{proof}

\begin{lemma}\label{lem:galois-action-stabilizer-general}
Let $k$ be a field of characteristic zero, let $G$ be an algebraic
$k$-group, let $H \subset G$ be a finite $k$-subgroup, and let $P$ be a
right $G$-torsor over $k$. Set $Q \coloneqq P/H$, so that $Q$ is a
homogeneous space under $G_P \coloneqq \Aut_G(P)$. Let
$\Bar{v} \in Q(\kbar)$, and let $H_{\Bar{v}}$ be the stabilizer of
$\Bar{v}$ for the action of $(G_P)(\kbar)$ on $Q_{\kbar}$. Then there is
an isomorphism of $\Gamma_k$-modules
\[
        H_{\Bar{v}}^{\mathrm{ab}} \cong H^{\mathrm{ab}}(\kbar),
\]
where $\Gamma_k$ acts on $H_{\Bar{v}}^{\mathrm{ab}}$ via
\eqref{eq:outer-action}, applied to the homogeneous space $Q$ under $G_P$.
\end{lemma}

\begin{proof}
The $k$-group $G_P=\Aut_G(P)$ acts on $P$ on the left and this action
commutes with the right $G$-action on $P$. In particular, it commutes with
the restricted right $H$-action, and therefore descends to an action of
$G_P$ on $Q=P/H$.

Fix $\bar p \in P(\kbar)$ lying above $\Bar{v}$. For every
$g\in G(\kbar)$, let $\varphi_g \in G_P(\kbar)$ be the $G_{\kbar}$-equivariant automorphism of $P_{\kbar}$ defined by $\varphi_g(\bar p g')=\bar p g g'$ for all $g'\in G(\kbar)$. Since $\Bar{p}$ lifts $\Bar{v}$, we have $\varphi_g\in H_{\Bar{v}}$ if and only if $g\in H(\kbar)$. Therefore, sending $g\mapsto \varphi_g$ gives an isomorphism
$G(\kbar)\xrightarrow{\sim} G_P(\kbar)$ which restricts to an isomorphism $\varphi\colon H(\kbar)\xrightarrow{\sim} H_{\Bar{v}}$. Let $\varphi^{\mathrm{ab}}\colon
        H^{\mathrm{ab}}(\kbar)\xrightarrow{\sim}
        H_{\Bar{v}}^{\mathrm{ab}}$
be the induced isomorphism on abelianizations. It remains to show that
$\varphi^{\mathrm{ab}}$ is $\Gamma_k$-equivariant. 

Let $\sigma\in \Gamma_k$, and let $g_\sigma\in G(\kbar)$ be the unique element such that $\sigma(\bar p)=\bar p g_\sigma$. For every $g\in G(\kbar)$, the automorphism $\sigma(\varphi_g)$ sends $\bar p g'$ to $\bar p g_\sigma\sigma(g)g_\sigma^{-1}g'=\varphi_{g_\sigma\sigma(g)g_\sigma^{-1}}(\bar p g')$, and hence
\begin{equation}\label{eq:galois-action-stabilizer-1}
        \sigma(\varphi_g)=\varphi_{g_\sigma\sigma(g)g_\sigma^{-1}} \quad\text{ for every $g\in G(\kbar)$}.
\end{equation}

Since $\sigma(\Bar{v})=\bar p g_\sigma H$, the element
$\varphi_{g_\sigma^{-1}}\in G_P(\kbar)$ sends $\sigma(\Bar{v})$ to
$\Bar{v}$. Therefore, with the notation of \eqref{eq:outer-action}, the pair $(\varphi_{g_\sigma^{-1}},\sigma)$ belongs to the subgroup $(G_P)_{\Bar{v}}\subset G_P(\kbar)\rtimes \Gamma_k$ defined in \eqref{eq:outer-action}. It follows that the outer automorphism of $H_{\Bar{v}}$ corresponding to $\sigma$ is represented by the automorphism that sends $\psi\in H_{\Bar{v}}$ to $\varphi_{g_\sigma^{-1}}\sigma(\psi)\varphi_{g_\sigma}$. Combining this with \eqref{eq:galois-action-stabilizer-1}, we deduce for every $h\in H(\kbar)$ the equality \begin{equation}\label{eq:galois-action-stabilizer-2}
\varphi_{g_\sigma^{-1}}\sigma(\varphi_h)\varphi_{g_\sigma}=\varphi_{g_\sigma^{-1}}\varphi_{g_\sigma\sigma(h)g_\sigma^{-1}}\varphi_{g_\sigma}= \varphi_{\sigma(h)}\qquad\text{in $G_P(\kbar)$, hence in $H_{\Bar{v}}$.}
\end{equation}
From \eqref{eq:galois-action-stabilizer-2}, we obtain a commutative square
\[
\begin{tikzcd}
        H(\kbar) \arrow[r,"\varphi"] \arrow[d,"\sigma"']
        & H_{\Bar{v}} \arrow[d,
        "\varphi_{g_\sigma^{-1}}\sigma(-)\varphi_{g_\sigma}"] \\
        H(\kbar) \arrow[r,"\varphi"]
        & H_{\Bar{v}}.
\end{tikzcd}
\]
Passing to abelianizations, the right vertical arrow induces the
action of $\sigma$ on $H_{\Bar{v}}^{\mathrm{ab}}$, while the left
vertical arrow induces the natural Galois action on
$H^{\mathrm{ab}}(\kbar)$. Therefore
$\varphi^{\mathrm{ab}}$ is $\Gamma_k$-equivariant, as desired.
\end{proof}

\begin{proposition}\label{prop:vea-stabilizer-ab}
Let $k$ be a field of characteristic zero, let $n\geq 3$, let $A$ be a central
simple algebra of degree $n$ over $k$, and let $E$ be an elliptic curve over $k$. Let
$\Bar{v}\in V_E^A(\kbar)$, and let $H_{\Bar{v}}$ be the stabilizer of
$\Bar{v}$ for the action of $\SL(A)(\kbar)$ on $V_E^A(\kbar)$. Then we have
isomorphisms of $\Gamma_k$-modules
\[
H_{\Bar{v}}^{\mathrm{ab}}\cong \cdual{H}_{\Bar{v}}^{\mathrm{ab}}\cong E[n](\kbar).
\]
\end{proposition}

\begin{proof}
By Proposition~\ref{prop:pgla-stabilizer}, $V_E^A$ is in the situation of Lemma~\ref{lem:galois-action-stabilizer-general}, with $G=\Aut(\bP_{E,n})$,  $H=E[n]$ and $P=R_E^A$. Therefore, letting $S_{\Bar{v}}\subset \PGL(A)(\kbar)$ be the stabilizer of $\Bar{v}$ for the $\PGL(A)(\kbar)$-action, we have $S_{\Bar{v}}\cong E[n](\kbar)$ as $\Gamma_k$-modules.  The central exact sequence \eqref{eq:mun-sla-pgla} restricts to an exact sequence
\[
        1\longrightarrow \mu_n(\kbar)
        \longrightarrow H_{\Bar{v}}
        \longrightarrow S_{\Bar{v}}
        \longrightarrow 1.
\]
By Proposition~\ref{picard-scheme-twisted}, the group $H_{\Bar{v}}$ is isomorphic to $\mathrm{He}_n(\kbar)$. By Lemma~\ref{theta-lemma}(a), the commutator subgroup of $H_{\Bar{v}}$ is exactly the central subgroup $\mu_n(\kbar)$. Therefore the projection $H_{\Bar{v}}\to S_{\Bar{v}}$ induces an isomorphism $H_{\Bar{v}}^{\mathrm{ab}}\xrightarrow{\sim}S_{\Bar{v}}$.
This isomorphism is compatible with the outer Galois actions because it is induced by the $k$-homomorphism $\SL(A)\to\PGL(A)$. Therefore $H_{\Bar{v}}^{\mathrm{ab}}\cong E[n](\kbar)$ as $\Gamma_k$-modules.

Finally, the Weil pairing is perfect and $\Gamma_k$-equivariant, so it gives an isomorphism of $\Gamma_k$-modules $E[n](\kbar)\cong \operatorname{Hom}_{\bZ}(E[n](\kbar),\mu_n)$. 
Thus $\cdual{H}_{\Bar{v}}^{\mathrm{ab}}\cong E[n](\kbar)$, as desired.
\end{proof}

        \subsection{Local points on \texorpdfstring{$V_E^A$}{VEA} in the presence of full level structure}

        Let $k$ be a number field, let $n\geq 3$ be an integer, let $A$ be a central simple algebra of degree $n$ over $k$, and let $E$ be an elliptic curve over $k$. In order to apply Corollary~\ref{brauer-corollary} to the $\SL(A)$-homogeneous space $V_E^A$, one must first choose a finite Galois extension $L/k$ such that $V_E^A(L)\neq\emptyset$. The next proposition, which refines~\cite[Theorem~C]{antieau-auel}, shows that this is possible if $E_L$ has full level $n$-structure. In the proof of Proposition~\ref{brauer-vanishing}, we will apply this over the field $L=k(E[2n])$ to show that $\Br_{\nr}(V_E^A)=\Br_0(V_E^A)$. 
		
		\begin{proposition}\label{full-level-structure}
			Let $k$ be a field, let $n\geq 3$ be an integer invertible in $k$, let $A$ be a cyclic central simple algebra of degree $n\geq 3$ over $k$, and let $E$ be an elliptic curve over $k$ such that $E[n]\cong \bZ/n\bZ\times_k \mu_n$. Then $V_E^A(k)\neq\emptyset$.
		\end{proposition}

		\begin{proof}
		Fix a $k$-group isomorphism $\varphi\colon \bZ/n\bZ\times_k\mu_n\xrightarrow{\sim}E[n]$, and let $P\in E[n](k)$ be the image of a generator of $\bZ/n\bZ\times_k\{1\}$. Let $h\colon E[n]\rightarrow\PGL_n$ be the embedding induced by the degree $n$ invertible sheaf $\LL\coloneqq \OO(0_E+P+\cdots+(n-1)P)$; see \eqref{eq:h_n}. By~\cite[Proposition~2.14]{antieau-auel}, there exists $\sigma\in H^1(k,\bZ/n\bZ)$ (which depends only on $E$, $n$, and $\varphi$) such that the induced map
        \[Ob\colon H^1(k,\bZ/n\bZ)\times H^1(k,\mu_n) \xlongrightarrow{\varphi_*}H^1(k,E[n])\xlongrightarrow{h_*} H^1(k,\PGL_n)\lhook\joinrel\longrightarrow\Br(k)[n]\] is given by $Ob(\chi,\lambda)=(\chi+\sigma)\cup \lambda$.
            
		Since the Brauer class of $A$ is cyclic, there exist $\chi\in H^1(k,\bZ/n\bZ)$ and $\lambda \in H^1(k,\mu_n)$ such that $[A]=\chi\cup \lambda$ in $\Br(k)$.
			Let $D$ be a right $E[n]$-torsor with class $(\chi-\sigma,\lambda)\in H^1(k,E[n])=H^1(k,\bZ/n\bZ)\times H^1(k,\mu_n)$, and let $D_h=D\times^{E[n]}\PGL_n$ be the $\PGL_n$-torsor obtained from $D$ by extension of structure group along $h$. By the choice of $D$, the image of the class $[D_h] \in H^1(k,\PGL_n)$ in $\Br(k)$ is
\[
Ob(\chi-\sigma,\lambda)=\chi\cup\lambda=[A].
\]
Thus $D_h$ is the $\PGL_n$-torsor corresponding to $A$, that is, with the
notation of Section~\ref{twisted-case}, $D_h\cong P^A$. Hence
\[
D\times^{E[n]}\bP^{n-1}_k
\cong
D_h\times^{\PGL_n}\bP^{n-1}_k
\cong
\SB(A).
\]
       On the other hand, $C=(D\times_kE)/E[n]$ is a genus $1$ curve such that $\Pic^0_{C/k}\cong E$. Embedding $E$ in $\bP^{n-1}_k$ as an elliptic normal curve via $\LL$ gives an $E[n]$-equivariant inclusion $\iota\colon E\hookrightarrow \bP^{n-1}_k$; see \eqref{eq:iota_n} and \eqref{eq:iota-h}. Twisting $\iota$ by the $E[n]$-torsor $D$ yields the desired twisted elliptic normal curve $C\hookrightarrow \SB(A)$. We conclude that $V_E^A(k)\neq\emptyset$.
		\end{proof}

		\subsection{Vanishing of the Brauer--Manin obstruction and proof of Theorem~\ref{local-global-thm}(1)}
		
		\begin{lemma}\label{sln-cyc}
			Let $n,d\geq 1$ be integers, and let $M_n$ be the $\SL_2(\bZ/dn\bZ)$-module $(\bZ/n\bZ)^2$, on which $\SL_2(\bZ/dn\bZ)$ acts by matrix multiplication via the natural reduction map $\SL_2(\bZ/dn\bZ)\to \SL_2(\bZ/n\bZ)$. Define the elements $\gamma_1,\gamma_2\in \SL_2(\bZ/dn\bZ)$ by 
			\[\gamma_1=\begin{pmatrix}
				1 & 1 \\
				0 & 1
			\end{pmatrix},\qquad \gamma_2=\begin{pmatrix}
				1 & 0 \\
				1 & 1
			\end{pmatrix}.\] 
			Then the pullback map \[(\gamma_1^*,\gamma_2^*)\colon H^1(\SL_2(\bZ/dn\bZ),M_n)\longrightarrow H^1(\bZ ,M_n)\oplus H^1(\bZ,M_n)\] is injective.
		\end{lemma}
		
		\begin{proof}  
			As $\bZ$ is a Euclidean domain, $\SL_2(\bZ)$ is generated by elementary matrices, and hence it is generated by the matrices $\begin{pmatrix}
				1 & 1 \\
				0 & 1
			\end{pmatrix}$ and $\begin{pmatrix}
				1 & 0 \\
				1 & 1
			\end{pmatrix}$. The reduction map $\SL_2(\bZ)\to \SL_2(\bZ/dn\bZ)$ being surjective (see e.g. \cite[Exercise 1.2.2(b)]{diamond2005first}), we deduce that $\SL_2(\bZ/dn\bZ)$ is generated by $\gamma_1$ and $\gamma_2$. 
			
			Let $F$ be the free group on two generators $c_1$ and $c_2$, and consider the surjective homomorphism $f\colon F\to \SL_2(\bZ/dn\bZ)$ given by $c_1\mapsto \gamma_1$ and $c_2\mapsto \gamma_2$. We view $M_n$ as an $F$-module via $f$. Since $M_n^{F}=0$, the Mayer--Vietoris sequence for $F=\langle c_1\rangle * \langle c_2 \rangle$ and the $F$-module $M_n$ reads
			\[0\longrightarrow M_n^{\langle c_1\rangle}\oplus M_n^{\langle c_2\rangle}\longrightarrow M_n\longrightarrow H^1(F,M_n)\xrightarrow{(c_1^*,c_2^*)} H^1(\bZ,M_n)\oplus H^1(\bZ,M_n)\longrightarrow 0\]
			see \cite[Exercise III.6(b)]{brown1994cohomology}.
			Since $M_n^{\langle c_1\rangle }=\langle (1,0)\rangle$ and $M_n^{\langle c_2\rangle}=\langle (0,1)\rangle$, the map $M_n^{\langle c_1\rangle }\oplus M_n^{\langle c_2\rangle}\to M_n$ is surjective, and hence the map $(c_1^*,c_2^*)$ is injective. We have a commutative square
			\[
			\begin{tikzcd}
				H^1(\SL_2(\bZ/dn\bZ),M_n) \arrow[r,"\text{$(\gamma_1^*,\gamma_2^*)$}"] \arrow[d,hook]  & H^1(\bZ,M_n)\oplus H^1(\bZ,M_n) \arrow[d,equal] \\
				H^1(F,M_n) \arrow[r,hook,"\text{$(c_1^*,c_2^*)$}"] & H^1(\bZ,M_n)\oplus H^1(\bZ,M_n),
			\end{tikzcd}
			\]
			where the left vertical arrow is an inflation map, and so it is injective by the inflation-restriction sequence. We conclude that the top horizontal map is injective, as desired.
		\end{proof}

		\begin{proposition}\label{brauer-vanishing}
			Let $k$ be a field of characteristic zero, let $A$ be a cyclic central simple algebra of degree $n\geq 3$ over $k$, and let $E$ be an elliptic curve over $k$ such that the image of $\Gamma_k\to \GL(E[2n](\kbar))$ contains $\SL(E[2n](\kbar))$. Then 
			\[\Br_{\nr}(V_E^A)=\Br_{1,\nr}(V_E^A)=\Br_0(V_E^A).\]
		\end{proposition}

		\begin{proof}
        Let $\Bar{v}\in V_E^A(\kbar)$ be a geometric point, and let $H_{\Bar{v}}\subset \SL(A)(\kbar)$ be the stabilizer of $\Bar{v}$. By assumption and Proposition~\ref{prop:vea-stabilizer-ab}, the image of $\Gamma_k\to \GL(H^{\mathrm{ab}}_{\Bar{v}})$ contains $\SL(H^{\mathrm{ab}}_{\Bar{v}})$.
        
			We first show that $\Br_{\nr}(V_E^A)=\Br_{1,\nr}(V_E^A)$. By Proposition~\ref{picard-scheme-twisted}, we have a group isomorphism $H_{\Bar{v}}\cong \mathrm{He}_n(\kbar)$. By a theorem of Bogomolov \cite{bogomolov1987brauer} (see also \cite[Theorem 7.1]{colliot2007rationality}), the group $\Br_{\nr}((V_E^A)_{\kbar})$ is isomorphic to the Bogomolov multiplier of $H_{\Bar{v}}$, and hence it is trivial by Lemma~\ref{theta-lemma}(d). Thus $\Br_{\nr}(V_E^A)=\Br_{1,\nr}(V_E^A)$, as desired.

            We now show that $\Br_{1,\nr}(V_E^A)=\Br_0(V_E^A)$ when $k$ contains a primitive $2n$th root of unity $\zeta_{2n}$. Consider the finite Galois extension $L\coloneqq k(E[2n])$ of $k$. Since $k$ contains $\zeta_{2n}$, the cyclotomic character $\Gamma_k\to (\bZ/2n\bZ)^\times$ is trivial. Moreover, because the determinant of the Galois action on $E[2n](\kbar)$ is the mod-$2n$ cyclotomic character, the image of $\Gamma_k \to \GL(E[2n](\kbar))$ is contained in $\SL(E[2n](\kbar))$ and hence is equal to $\SL(E[2n](\kbar))$. By Proposition~\ref{full-level-structure}, we deduce $V_E^A(L)\neq\emptyset$. Since $k(E[n])\subset L$, the subgroup $\Gamma_L\subset \Gamma_k$ acts trivially on $E[n](\kbar)$, and hence, by Proposition~\ref{prop:vea-stabilizer-ab}, it acts trivially on $H_{\Bar{v}}^{\mathrm{ab}}$. By Lemma~\ref{theta-lemma}(c), the subgroup $\Gamma_L$ acts on $H_{\Bar{v}}$ by inner automorphisms, that is, the restriction of $\Gamma_k\to \mathrm{Out}(H_{\Bar{v}})$ to $\Gamma_L$ is trivial. It follows that the outer $\Gamma_k$-action on $H_{\Bar{v}}$ factors through $\Gal(L/k)$. Moreover, by Lemma~\ref{theta-lemma}(b), the exponent of $H_{\Bar v}$ divides $2n$, and under our current assumptions $\zeta_{2n}\in k$. Thus Corollary~\ref{brauer-corollary} applies.
			
			Let $\beta\in H^1(\Gal(L/k),\cdual{H}_{\Bar{v}}^{\mathrm{ab}})$ be such that $\mathrm{Inf}(\beta)$ is the image of an element of $\Br_{1,\nr}(V_E^A)$. We must show that $\beta=0$. Let $\bar u_1,\bar u_2\in H_{\Bar{v}}^{\mathrm{ab}}$ be a $(\bZ/n\bZ)$-basis, and choose lifts $u_1,u_2\in H_{\Bar{v}}$. Choose a $\bZ/2n\bZ$-basis $w_1,w_2$ of $E[2n](\Bar k)$ such that $2w_1,2w_2\in E[n](\Bar k)$ correspond, under the isomorphism $E[n](\Bar k)\cong H_{\Bar v}^{\mathrm{ab}}$, to $\bar u_1,\bar u_2$. With respect to the basis $w_1,w_2$, we identify $\Gal(L/k)$ with $\SL_2(\bZ/2n\bZ)$. Let $g_1,g_2\in\Gal(L/k)$ be the elements corresponding to the matrices
\[
\begin{pmatrix}
1 & 1 \\
0 & 1
\end{pmatrix},
\qquad
\begin{pmatrix}
1 & 0 \\
1 & 1
\end{pmatrix},
\]
respectively. Then, after reduction modulo $n$, their actions on $H_{\Bar v}^{\mathrm{ab}}$ are given by
\[
g_1(\bar u_1)=\bar u_1,\quad g_1(\bar u_2)=\bar u_1\bar u_2,\qquad g_2(\bar u_1)=\bar u_1\bar u_2,\quad g_2(\bar u_2)=\bar u_2.
\]
We show that $g_1^*\beta=0$. Choose an automorphism $\varphi_1\in\Aut(H_{\Bar{v}})$ representing the outer
action of $g_1$ on $H_{\Bar{v}}$. Its induced action on $H_{\Bar{v}}^{\mathrm{ab}}$ sends
$\bar u_1$ to $\bar u_1$ and $\bar u_2$ to $\bar u_1\bar u_2$. By
Lemma~\ref{theta-lemma}(e), the element $\varphi_1(u_1)$ is conjugate to
$u_1$. Therefore the conjugacy class of $u_1$ is fixed by $g_1$, and hence
defines an element of $(H_{\Bar{v}}/\mathrm{conj})^{g_1}$ whose image in
$(H_{\Bar{v}}^{\mathrm{ab}})^{g_1}$ is $\bar u_1$. Hence the image of the natural map $(H_{\Bar{v}}/\mathrm{conj})^{g_1}\to (H_{\Bar{v}}^{\mathrm{ab}})^{g_1}$ contains $\bar u_1$, and therefore generates the group $(H_{\Bar{v}}^{\mathrm{ab}})^{g_1}$. Now Corollary~\ref{brauer-corollary} implies that $g_1^*\beta=0$, as desired. Similarly, $g_2^*\beta=0$. Since $k$ contains $\zeta_{2n}$, the Weil pairing identifies
$\cdual{H}_{\Bar{v}}^{\mathrm{ab}}$ with $H_{\Bar{v}}^{\mathrm{ab}}$ as a
$\Gal(L/k)$-module. After choosing the basis
$\bar u_1,\bar u_2$, the action of $\Gal(L/k)$ on this module is the standard
action of $\SL_2(\bZ/2n\bZ)$ on $(\bZ/n\bZ)^2$ via
reduction modulo $n$. By Lemma~\ref{sln-cyc} (applied to $d=2$), the fact that $g_1^*\beta=0$ and $g_2^*\beta=0$ now implies that $\beta=0$. This proves Proposition~\ref{brauer-vanishing} under the assumption that $\zeta_{2n}\in k$.

			Now let $k$ be an arbitrary field of characteristic zero, and let $\alpha\in H^1(k,\cdual{H}_{\Bar{v}}^{\mathrm{ab}})$ be a class in the image of $\Br_{1,\nr}(V_E^A)\to H^1(k,\cdual{H}_{\Bar{v}}^{\mathrm{ab}})$. We must show that $\alpha=0$. Define $k'=k(\zeta_{2n})$. Since $\Gamma_{k'}$ is the kernel of the mod-$2n$ cyclotomic character, the image of $\Gamma_{k'}$ in $\GL(E[2n](\Bar k))$ is equal to $\SL(E[2n](\Bar k))$. Then $\alpha_{k'}$ belongs to the image of $\Br_{1,\nr}((V_E^A)_{k'})\to H^1(k',\cdual{H}_{\Bar{v}}^{\mathrm{ab}})$ and hence, by the previous part of the proof, it is trivial. We have the inflation-restriction sequence 
			\[0\longrightarrow H^1(\Gal(k'/k),(\cdual{H}^{\mathrm{ab}}_{\Bar{v}})^{\Gamma_{k'}})\longrightarrow H^1(k,\cdual{H}_{\Bar{v}}^{\mathrm{ab}})\longrightarrow H^1(k',\cdual{H}_{\Bar{v}}^{\mathrm{ab}}).\]
			The image of $\Gamma_{k'}\to \GL(\cdual{H}_{\Bar{v}}^{\mathrm{ab}})$ is equal to $\SL(\cdual{H}_{\Bar{v}}^{\mathrm{ab}})$. Therefore $(\cdual{H}^{\mathrm{ab}}_{\Bar{v}})^{\Gamma_{k'}}=0$, and hence the restriction map $H^1(k,\cdual{H}^{\mathrm{ab}}_{\Bar{v}})\to H^1(k',\cdual{H}^{\mathrm{ab}}_{\Bar{v}})$ is injective. We conclude that $\alpha=0$, as desired.
		\end{proof}

		\begin{proof}[Proof of Theorem~\ref{local-global-thm}(1)]	
			Write $X=\SB(A)$, where $A$ is a central simple algebra of degree $n$ over $k$. By assumption, the product of the $V_E^A(k_v)$, for all places $v$ of $k$, is not empty. By Proposition~\ref{brauer-vanishing}, we have $\Br_{\nr}(V_E^A)=\Br_0(V_E^A)$, and so the Brauer--Manin set of $V_E^A$ coincides with the product of the $V_E^A(k_v)$, for all places $v$ of $k$, and hence in particular it is not empty. By Theorem~\ref{demeio-thm}, the set $V_E^A(k)$ is dense in the Brauer--Manin set of $V_E^A$ with respect to the $v$-adic topologies, and hence it is also not empty.
		\end{proof}

		\subsection{Proof of Theorem~\ref{local-global-thm}(2)}

The following observation is standard.
        
\begin{lemma}\label{cm-large-image-exclusion}
Let $k$ be a field of characteristic zero, let $m\geq 3$ be an integer, and let
$E$ be an elliptic curve over $k$ with $j$-invariant $0$ or $1728$. Then the image of $\Gamma_k\to \GL(E[m](\kbar))$ does not contain $\SL(E[m](\kbar))$.
\end{lemma}

\begin{proof}
    Choose a $\bZ/m\bZ$-basis of $E[m](\kbar)$ in which a complex multiplication automorphism of $E_{\kbar}$ is represented by
\[
        \begin{pmatrix}0&-1\\ 1&0\end{pmatrix}
        \quad\text{if } \jj(E)=1728
        \qquad \text{or} \qquad
        \begin{pmatrix}0&-1\\ 1&-1\end{pmatrix}
        \quad\text{if } \jj(E)=0.
\]
The action of $\Gamma_k$ on $E[m](\kbar)$ normalizes the image of
$\Aut(E_{\kbar})$ in $\GL(E[m](\kbar))$, and hence normalizes the cyclic subgroup generated by the chosen complex multiplication automorphism. After choosing the above basis, the image of $\Gamma_k$ in $\GL_2(\bZ/m\bZ)$ is contained in the normalizer of the corresponding Cartan subgroup
\[
C_{1728}(m)=
\left\{
\begin{pmatrix}a&-b\\ b&a\end{pmatrix}
:
a,b\in \bZ/m\bZ,\ 
a^2+b^2\in(\bZ/m\bZ)^\times
\right\},
\]
or
\[
C_0(m)=
\left\{
\begin{pmatrix}a&-b\\ b&a-b\end{pmatrix}
:
a,b\in \bZ/m\bZ,\ 
a^2-ab+b^2\in(\bZ/m\bZ)^\times
\right\},
\]
respectively. 
The centralizer of the chosen CM automorphism is contained in the corresponding Cartan subgroup, whose order is evidently at most $m^2$. The normalizer maps to the automorphism group of the cyclic subgroup generated by the CM automorphism, which has order at most $2$, and the kernel of this homomorphism is the centralizer of the CM automorphism. It follows that the normalizer has order at most $2m^2$. On the other hand, for $m\geq 4$ one has
\[
|\SL_2(\bZ/m\bZ)|
=
m^3\prod_{\ell\mid m}\left(1-\ell^{-2}\right)
\geq \frac{m^3}{\zeta(2)}
=\frac{6m^3}{\pi^2}
>2m^2.
\]
For $m=3$, one has $|\SL_2(\bZ/3\bZ)|=24>18=2\cdot 3^2$. Hence, for all $m\geq 3$, the image of $\Gamma_k\to \GL(E[m](\kbar))$ does not contain $\SL(E[m](\kbar))$, as desired.
\end{proof}

    In view of Lemma~\ref{cm-large-image-exclusion}, the large-image hypothesis in Theorem~\ref{local-global-thm}(2) excludes the cases $j_0=0$ and $j_0=1728$. Hence $\Aut_{E/k}\cong\mu_2$, and all twists of $E$ with $j$-invariant $j_0$ are quadratic twists. We will deduce Theorem~\ref{local-global-thm}(2) from Theorem~\ref{local-global-thm}(1) by twisting $E$ by a suitable quadratic character. The next two lemmas serve to ensure that the open-image condition is preserved under twisting.
        
		\begin{lemma}\label{image-of-twist-group-theory}
			Let $n\geq 2$ be an integer, let $f\colon G\to \SL_2(\bZ/n\bZ)$ be a surjective homomorphism, let $\chi\colon G\to \{\pm 1\}$ be a homomorphism, and let $\chi\otimes f\colon G\to \SL_2(\bZ/n\bZ)$ be the homomorphism sending $g$ to $(\chi(g)\mathrm{I}_2)\cdot f(g)$, where $\mathrm{I}_2\in \SL_2(\bZ/n\bZ)$ is the identity matrix. Then $\chi\otimes f$ is surjective.
		\end{lemma}
		
		\begin{proof}
			Let $H\coloneqq \ker(\chi)\subset G$. Then $[G:H]\leq 2$ and the restrictions of $\chi\otimes f$ and $f$ to $H$ coincide. Thus $(\chi\otimes f)(H)=f(H)$ has index $\leq 2$ in $\SL_2(\bZ/n\bZ)$. 
			
			Any index $2$ subgroup of $\SL_2(\bZ/n\bZ)$ contains $[\SL_2(\bZ/n\bZ),\SL_2(\bZ/n\bZ)]$. Since the reduction homomorphism $\SL_2(\bZ)\to \SL_2(\bZ/n\bZ)$ is surjective, $\SL_2(\bZ/n\bZ)^{\mathrm{ab}}$ is a quotient of $\SL_2(\bZ)^{\mathrm{ab}}$. It is well known (see e.g. \cite[p. 81]{serre1980trees}), that $\SL_2(\bZ)$ admits the presentation
			\[\SL_2(\bZ)=\langle S, T\, |\, S^4=1, (ST)^3=S^2\rangle,\qquad S=\begin{pmatrix}
				0 & 1\\
				-1 & 0
			\end{pmatrix},\quad T=\begin{pmatrix}
				1 & 0\\
				1 & 1
			\end{pmatrix}.\] In particular, $\SL_2(\bZ)^{\mathrm{ab}}$ is cyclic of order $12$, generated by the coset of $T$. Moreover, the coset of $-\mathrm{I}_2$ in $\SL_2(\bZ)^{\mathrm{ab}}$ is equal to the coset of $T^6$, and hence it has order $2$. As $\SL_2(\bZ/n\bZ)$ is a quotient of $\SL_2(\bZ)$, it follows that there exists at most one subgroup of index $2$ in $\SL_2(\bZ/n\bZ)$, and if it exists, it contains $-\mathrm{I}_2$.
			
			If $\SL_2(\bZ/n\bZ)$ has no subgroup of index $2$, then $f(H)=\SL_2(\bZ/n\bZ)$ and the proof is complete. If $\SL_2(\bZ/n\bZ)$ has a subgroup of index $2$, then it must be $f(H)$. We deduce that $-\mathrm{I}_2\in f(H)$, and hence the image of $\chi\otimes f$ contains $-\mathrm{I}_2$. Since $f$ is surjective, this implies that $\chi\otimes f$ is also surjective.
		\end{proof}

\begin{lemma}\label{image-of-twist}
Let $k$ be a field, let $n\geq 1$ be an integer such that $2n$ is invertible in $k$, let $E$ be an elliptic curve over $k$, and let $E'$ be a quadratic twist of $E$. If the image of $\rho\colon \Gamma_k\to \GL(E[2n](\kbar))$ contains $\SL(E[2n](\kbar))$, then the image of $\rho'\colon \Gamma_k\to \GL(E'[2n](\kbar))$ contains $\SL(E'[2n](\kbar))$.
\end{lemma}

\begin{proof}
Set $K=k(\zeta_{2n})$. By the Weil pairing, the homomorphisms $\rho$ and $\rho'$ restrict to
\[
\rho|_{\Gamma_K}\colon \Gamma_K\longrightarrow \SL(E[2n](\kbar)),
\qquad
\rho'|_{\Gamma_K}\colon \Gamma_K\longrightarrow \SL(E'[2n](\kbar)).
\]
We first claim that $\rho(\Gamma_K)=\SL(E[2n](\kbar))$. Indeed, by assumption the image of $\rho$ contains
$\SL(E[2n](\kbar))$. If $s\in \SL(E[2n](\kbar))$, choose
$\gamma\in\Gamma_k$ such that $\rho(\gamma)=s$. Since the determinant of
$\rho(\gamma)$ is the mod-$2n$ cyclotomic character and $\det(s)=1$, we have
$\gamma\in\Gamma_K$. Thus $\rho(\Gamma_K)=\SL(E[2n](\kbar))$, as desired.

Let $\chi\colon \Gamma_k\to\{\pm1\}$ be the quadratic character corresponding
to the twist $E'$. Choose an isomorphism $\iota\colon E_{\kbar}\xrightarrow{\sim} E'_{\kbar}$, and let $\iota_{2n}\colon E[2n](\kbar)\xrightarrow{\sim} E'[2n](\kbar)$ be the induced isomorphism. Under this identification, the two Galois
representations are related by the identity $\iota_{2n}^{-1}\rho'(\gamma)\iota_{2n}
=
\chi(\gamma)\rho(\gamma)$ for all $\gamma\in\Gamma_k$, where we view $\chi(\gamma)$ as the scalar matrix $\pm \mathrm{I}_2$. 
Restricting to $\Gamma_K$, we obtain
\[
\iota_{2n}^{-1}\rho'(\Gamma_K)\iota_{2n}
=
(\chi|_{\Gamma_K}\otimes \rho|_{\Gamma_K})(\Gamma_K),
\]
where by definition $(\chi|_{\Gamma_K}\otimes \rho|_{\Gamma_K})(\gamma)
=
\chi(\gamma)\rho(\gamma)$ for every $\gamma\in\Gamma_K$. Since $\rho(\Gamma_K)=\SL(E[2n](\kbar))$, Lemma~\ref{image-of-twist-group-theory} implies that $\chi|_{\Gamma_K}\otimes \rho|_{\Gamma_K}$ is also surjective onto $\SL(E[2n](\kbar))$. Hence
$\rho'(\Gamma_K)=\SL(E'[2n](\kbar))$, and in particular, the image of $\rho'$ contains $\SL(E'[2n](\kbar))$.
\end{proof}

		\begin{proof}[Proof of Theorem~\ref{local-global-thm}(2)]	
       Write $X=\SB(A)$, where $A$ is a central simple algebra of degree $n$ over $k$. By Lemma~\ref{cm-large-image-exclusion}, we may assume that $j_0\notin\{0,1728\}$. Let $S$ be the finite set of places $v$ of $k$ such that $A_{k_v}$ is not split. For every $v\in S$, let $C_v\subset \SB(A_{k_v})$ be a twisted elliptic normal curve such that $\jj(\Pic^0_{C_v/k_v})=j_0$, and define $E_v\coloneqq \Pic^0_{C_v/k_v}$. Then $\jj(E_v)=j_0$ and $C_v$ defines a $k_v$-point of $V_{E_v}^{A_{k_v}}$ (see \eqref{vea-points}), so that $V_{E_v}^{A_{k_v}}(k_v)\neq\emptyset$. Since $\jj(E)=j_0\notin\{ 0, 1728\}$, the automorphism group scheme $\Aut_{E/k}$ is isomorphic to $\mu_2$. For every $v\in S$, the elliptic curves $E_v$ and $E_{k_v}$ have the same $j$-invariant and hence become isomorphic over $\kbar_v$, so that $E_v$ defines an element in $H^1(k_v,\Aut_{E_{k_v}/k_v})=H^1(k_v,\mu_2)$. By weak approximation for $\bA^1_k$, the natural map $
k^\times/k^{\times 2}\to
\prod_{v\in S} k_v^\times/k_v^{\times 2}$ is surjective. Equivalently, by Kummer theory, the natural map $H^1(k,\mu_2)\to \prod_{v\in S}H^1(k_v,\mu_2)$ 
is surjective. It follows that there exists a quadratic twist $E'$ of $E$ such that $\jj(E')=j_0$ and $(E')_{k_v}\cong E_v$ for every $v\in S$. 

By assumption, the image of $\Gamma_k\to \GL(E[2n](\kbar))$ contains $\SL(E[2n](\kbar))$. By Lemma~\ref{image-of-twist}, the image of $\Gamma_k\to \GL(E'[2n](\kbar))$ contains $\SL(E'[2n](\kbar))$. Moreover, we have $V_{E'_{k_v}}^{A_{k_v}}(k_v)\neq\emptyset$ for every place $v$ of $k$: indeed, if $v\in S$ then $V_{E'_{k_v}}^{A_{k_v}}\cong V_{E_v}^{A_{k_v}}$ has a $k_v$-point corresponding to $C_v$, and if $v\notin S$ then $A_{k_v}$ is split, so that $\SB(A_{k_v})\cong \bP^{n-1}_{k_v}$, and the elliptic curve $E'_{k_v}$ embeds as an elliptic normal curve in $\bP^{n-1}_{k_v}$ via any invertible sheaf of degree $n$. Thus $E'$ satisfies the assumption of Theorem~\ref{local-global-thm}(1), and hence $V_{E'}^A(k)\neq\emptyset$. As $\jj(E')=j_0$, we have a morphism $V_{E'}^A\to V_{j_0}^A$ (see \eqref{bigdiagram-twisted}), and hence $V_{j_0}^A(k)\neq \emptyset$, as desired.
\end{proof}

\section{Proof of Theorem~\ref{weak-approximation-thm}}
\label{section-fixed-j-inv}

        In this section, we study the Brauer--Manin obstruction to weak approximation on the Hilbert scheme of twisted elliptic normal curves, leading up to a proof of Theorem~\ref{weak-approximation-thm}.
		
		Let $k$ be a number field. Recall that a $k$-variety $X$ is said to satisfy \emph{weak approximation} if for every finite set $S$ of places of $k$, the image of the diagonal map $X(k)\to \prod_{v \in S} X(k_v)$ is dense in the product of the $v$-adic topologies. If $X(k_v)\neq\emptyset$ for every place (v) of $k$, this is equivalent to the density of the diagonal image of $X(k)$ in $\prod_{v\in\Omega_k}X(k_v)$ equipped with the product topology.
		
		\begin{lemma}\label{lemma-w}
			Let $k$ be a number field, let $A$ be a central simple algebra of degree $n\geq 3$ over $k$, let $V\coloneqq  V(n)^A-(f^A)^{-1}(\{0,1728\})$ (cf. \eqref{bigdiagram-twisted}), let $t$ be the standard coordinate on $\bA^1_k$, let $U\subset \bA^1_k$ be a dense open subscheme, let $\sE\to U$ be a family of elliptic curves, and consider a Cartesian square
			\[
			\begin{tikzcd}
				W \arrow[r,"\varphi"] \arrow[d,"\pi"] & V \arrow[d,"(J^A)|_V"] \\
				U \arrow[r,"s"] & \sM_{1,1}, 
			\end{tikzcd}
			\]
			where $s$ is the morphism corresponding to $\sE$. Assume that the image of $\Gamma_{k(t)} \to \GL(\sE_{k(t)}[2n](\overline{k(t)}))$ contains $\SL(\sE_{k(t)}[2n](\overline{k(t)}))$. 
			Then $\Br_{\nr}(W)=\Br_0(W)$. Furthermore, if $W(k_v)\neq\emptyset$ for every $v\in \Omega_k$, then $W$ satisfies weak approximation.
		\end{lemma}
		
		\begin{proof}
			By Hironaka's resolution of singularities, there exist a smooth projective compactification $X$ of $W$ and a morphism $\Bar{\pi}\colon X\to \bP^1_k$ extending $\pi$. The equality $\Br_{\nr}(W)=\Br_0(W)$ is equivalent to $\Br(X)=\Br_0(X)$. In order to show that $\Br(X)=\Br_0(X)$, recall from \cite[Definition 11.1.11]{colliot2021brauer} the definition of the vertical Brauer group 
            \[\Br_{\mathrm{vert}}(X/\bP^1_k)\coloneqq \{\beta\in \Br(X)\,\mid\, \beta|_{X_{k(t)}}\in\mathrm{Im}(\Br(k(t))\to \Br(X_{k(t)}))\}.\]
            Since $X_{k(t)}$ is isomorphic to a smooth projective compactification of $V_{\sE_{k(t)}}^{A_{k(t)}}$ over $k(t)$, by Proposition~\ref{brauer-vanishing} we have $\Br(X_{k(t)})=\Br_0(X_{k(t)})$, that is, the pullback map $\Br(k(t))\to \Br(X_{k(t)})$ is surjective. It follows that $\Br(X)=\Br_{\mathrm{vert}}(X/\bP^1_k)$.
			
			Let $K\coloneqq k(\SB(A))$. Then $k$ is algebraically closed in $K$ and $A_K$ is split, so that by Proposition~\ref{prop:pgla-stabilizer} we have $V_{\sE_{k(t)}}^{A_{k(t)}}(K(t))\neq\emptyset$, that is, the morphism $\pi_K\colon W_K\to U_K$ has a rational section, and hence that $\Bar{\pi}_K\colon X_K\to \bP^1_K$ has a section. It follows from Lemma~\ref{skorobogatov-criterion} that all closed fibers of $\Bar{\pi}\colon X\to \bP^1_k$ are split. By \cite[Corollary 11.1.6(ii)]{colliot2021brauer}, we deduce that $\Br_{\mathrm{vert}}(X/\bP^1_k)=\Bar{\pi}^*\Br(\bP^1_k)$. On the other hand, we have $\Br(\bP^1_k)=\Br_0(\bP^1_k)$, so that $\Bar{\pi}^*\Br(\bP^1_k)=\Br_0(X)$. All in all, we have
            \[\Br(X)=\Br_{\mathrm{vert}}(X/\bP^1_k)=\Bar{\pi}^*\Br(\bP^1_k)=\Br_0(X),\]
            as desired. It follows that the Brauer--Manin set of $X$ is equal to the product of all the $X(k_v)$. 
			
			It remains to show that $X(k)$ is dense in its Brauer--Manin set. For this, it suffices to check that the assumptions of Theorem~\ref{harari-thm} are satisfied. The verification of these assumptions is given below; it closely parallels the corresponding part of the proof of Theorem~\ref{mainthm}.
            
			\begin{itemize} 
				\item[(i)] We already checked that the fiber of $\Bar{\pi}$ at $P$ is a split $k(P)$-scheme for every closed point $P\in \bA^1_k$. 
				\item[(ii)] By \eqref{bigdiagram-twisted}, the generic fiber of $\pi$ is isomorphic to $V_{\sE_{k(t)}}^{A_{k(t)}}$ over $k(t)$. Proposition~\ref{prop:pgla-stabilizer} now implies that $V_{\sE_{k(t)}}^{A_{k(t)}}$ is smooth over $k(t)$ and, since $A_{\kbar}$ is split, that $V_{\sE_{k(t)}}^{A_{k(t)}}(\kbar(t))\neq \emptyset$.
				\item[(iii)] The geometric generic fiber of $\Bar{\pi}$ is unirational, and hence its Picard group is torsion-free and its Brauer group is finite by \cite[Theorem 5.5.2 and (6) p. 347]{colliot2021brauer}.
				\item[(iv)] By Corollary~\ref{solvable-stabilizers}, for every $u\in U(k)$ the fiber $\pi^{-1}(u)=V_{\sE_u}^A$ is a homogeneous space under $\SL(A)$ with finite solvable stabilizer. By Theorem~\ref{demeio-thm}, if $(V_u^A)^c$ is a smooth projective compactification of $V_u^A$, then $(V_u^A)^c(k)$ is dense in the Brauer--Manin set of $(V_u^A)^c$.
                \end{itemize}
            By Theorem~\ref{harari-thm}, we conclude that $X(k)$ is dense in the product of the $X(k_v)$, that is, $X$ satisfies weak approximation. By a theorem of Kneser \cite[Proposition 13.2.3]{colliot2021brauer}, weak approximation is a birational invariant, so we deduce that $W$ satisfies weak approximation.
		\end{proof}

Finally, we arrive at the proof of Theorem~\ref{weak-approximation-thm}.
		
        \begin{proof}[Proof of Theorem~\ref{weak-approximation-thm}]
Consider the morphism $f^A\colon V(n)^A\to \bA^1_k$ of \eqref{eq:morphism-f} and the smooth open subscheme $V\coloneqq V(n)^A\smallsetminus(f^A)^{-1}(\{0,1728\})$ of $V(n)^A$. 
Since weak approximation is a birational invariant by a theorem of Kneser
\cite[Proposition~13.2.3]{colliot2021brauer}, in order to prove that
$V(n)^A$ satisfies weak approximation, it suffices to show that $V$
satisfies weak approximation.

Let $\Sigma$ be a finite set of places of $k$ and, for every
$v\in\Sigma$, let $C_v\subset \SB(A_{k_v})$ be a twisted elliptic normal curve corresponding to a point $[C_v]\in V(k_v)$. 
After enlarging $\Sigma$, we may assume that $A_{k_v}$ is split for
every place $v\notin\Sigma$. For every place added to $\Sigma$, choose
a point $[C_v]\in V(k_v)$. Such a point exists by
Proposition~\ref{nodal} and Remark~\ref{rmk:large}, since every central
simple algebra over a local field is cyclic and every local field is
large. For every $v\in\Sigma$, define $E_v\coloneqq \Pic^0_{C_v/k_v}$ and $j_v\coloneqq \jj(E_v)$. Since $[C_v]\in V(k_v)$, we have $j_v\notin\{0,1728\}$. Let $t$ be a coordinate on $\bA^1_k$, let $U\coloneqq \bA^1_k\smallsetminus\{0,1728\}$, let $x,y,z$ be homogeneous coordinates on $\bP^2_k$, and consider the
elliptic curve $\sE\subset \bP^2_k\times_k U$ over $U$ given by the homogeneous equation
\[y^2z+xyz=x^3-\frac{36}{t-1728}xz^2-\frac{1}{t-1728}z^3.\]
By \cite[Section~5.1, p.~47]{serre2008topics}:
\begin{itemize}
\item[(i)] the $j$-invariant morphism $U\to\bA^1_k$ is the
obvious inclusion, so that $\jj(\sE_{k(t)})=t$ and $\jj(\sE_\lambda)=\lambda$ for every $\lambda\in k\smallsetminus\{0,1728\}$;
\item[(ii)] the image of $\Gamma_{k(t)}\to\GL(\sE_{k(t)}[2n](\overline{k(t)}))$ contains $\SL(\sE_{k(t)}[2n](\overline{k(t)}))$. 
\end{itemize}

By (i), the elliptic curves $(\sE_{j_v})_{k_v}$ and $E_v$ are
quadratic twists of each other, for every $v\in\Sigma$. As in the proof
of Theorem~\ref{local-global-thm}(2), let $\chi_v\in H^1(k_v,\mu_2)$ be a character such that the twist of $(\sE_{j_v})_{k_v}$ by $\chi_v$ is isomorphic to $E_v$, for every $v\in\Sigma$. By weak approximation for $\bA^1_k$, the natural map $k^\times/k^{\times 2}
\to \prod_{v\in\Sigma}k_v^\times/k_v^{\times 2}$ is surjective. Hence, by Kummer theory, there exists $\chi\in H^1(k,\mu_2)$ whose image in $H^1(k_v,\mu_2)$ is equal to $\chi_v$, for every $v\in\Sigma$. Replacing $\sE$ by its twist by $\chi$, the new elliptic curve
$\sE\to U$ satisfies (i), because twisting does not change the $j$-invariant, satisfies (ii) by Lemma~\ref{image-of-twist}, and satisfies
\begin{itemize}
\item[(iii)] for every $v\in\Sigma$, we have an isomorphism $(\sE_{j_v})_{k_v}\cong E_v$ over $k_v$.
\end{itemize}

Consider the Cartesian square of Lemma~\ref{lemma-w} determined by the
family of elliptic curves $\sE\to U$:
\[
\begin{tikzcd}
W \arrow[r,"\varphi"] \arrow[d,"\pi"] &
V \arrow[d,"(J^A)|_V"]\\
U \arrow[r,"s"] &
\sM_{1,1}.
\end{tikzcd}
\]
By (iii), for every $v\in\Sigma$, we have $W_{j_v}\cong V_{E_v}^{A_{k_v}}$. For every $v\in\Sigma$, let $P_v\in W_{j_v}(k_v)$ be the point corresponding to the twisted elliptic normal curve $C_v\subset \SB(A_{k_v})$; then $\varphi(P_v)=[C_v]$.

We claim that $W(k_v)\neq\emptyset$ for every place $v$ of $k$. If
$v\in\Sigma$, this follows from the existence of $P_v$. Suppose that
$v\notin\Sigma$. Then $A_{k_v}$ is split. Choose any
$u_v\in U(k_v)$, so that $W_{u_v}\cong V_{\sE_{u_v}}^{A_{k_v}}$.
Since $\SB(A_{k_v})\cong\bP^{n-1}_{k_v}$, Proposition~\ref{prop:pgla-stabilizer} implies that $W_{u_v}(k_v)\neq\emptyset$, and hence $W(k_v)\neq\emptyset$.

It follows from Lemma~\ref{lemma-w} that $W$ satisfies weak
approximation. Therefore, there exists a sequence of points
$x_i\in W(k)$ whose images in $\prod_{v\in\Sigma}W(k_v)$ converge to $(P_v)_{v\in\Sigma}$. Since, for every $v\in\Sigma$, the
map $\varphi(k_v)\colon W(k_v)\to V(k_v)$ is continuous with respect to the $v$-adic topology, the sequence of points $\varphi(x_i)\in V(k)$ converges to $[C_v]$ in $V(k_v)$, for every $v\in\Sigma$. Thus $V$ satisfies weak approximation. Since $V$ is birational to $V(n)^A$, it follows that $V(n)^A$ satisfies weak approximation.
\end{proof}

\small
\providecommand{\bysame}{\leavevmode\hbox to3em{\hrulefill}\thinspace}
\providecommand{\MR}{\relax\ifhmode\unskip\space\fi MR }


\vspace{20pt}
\noindent
Benjamin Antieau\\
Northwestern University\\
\texttt{antieau@northwestern.edu}

\vspace{10pt}
\noindent
Asher Auel\\
Dartmouth College\\
\texttt{asher.auel@dartmouth.edu}

\vspace{10pt}
\noindent
Federico Scavia\\
CNRS and Université Sorbonne Paris Nord\\
\texttt{scavia@math.univ-paris13.fr}
        
\end{document}